\documentclass[a3paper, 11pt]{amsproc}

\usepackage{macros, enumitem}
\usepackage{amssymb, url}
\usepackage[dvipsnames]{xcolor}
\usepackage[theoremfont]{newtxtext}
\usepackage[bigdelims, varg, vvarbb]{newtxmath}
\usepackage{hyperref}
\newtheorem{theorem}{Theorem}[section]
\newtheorem{proposition}[theorem]{Proposition}
\newtheorem{corollary}[theorem]{Corollary}
\newtheorem{lemma}[theorem]{Lemma}
\theoremstyle{remark}
\newtheorem{remark}[theorem]{Remark}

\theoremstyle{definition}
\newtheorem{definition}[theorem]{Definition}

\title[EXOTIC: Exact solutions for min-max optimization]{EXOTIC: An Exact, Optimistic, Tree-Based Algorithm for Min-Max Optimization}
\author[C. Maheshwari]{Chinmay Maheshwari}
\author[C. Pimpalkhare]{Chinmay Pimpalkhare}
\author[D. Chatterjee]{Debasish Chatterjee}

\thanks{C.\ M.\ is with the Department of Electrical and Computer Engineering, and Data Science and AI
Institute at Johns Hopkins University, Baltimore, MD, USA (Email: chinmay\_maheshwari@jhu.edu). C.\ P.\ is with the Institute for Computational and Mathematical Engineering at Stanford University, Stanford, CA, USA (Email: cpimpalk@stanford.edu). D.\ C.\ is with the Center for Systems and Control, IIT Bombay, Mumbai, MH, India (Email: dchatter@iitb.ac.in)}



\begin{document}

\begin{abstract} 
Min–max optimization arise in many domains such as game theory, adversarial machine learning, robust optimization, control, and signal processing. In convex–concave min-max optimization, gradient-based methods are well understood and enjoy strong guarantees. However, in absence of convexity or concavity, existing approaches study convergence to an approximate saddle point or first-order stationary points, which may be arbitrarily far from global optima.

In this work, we present an algorithmic framework for computing the {global minimax value} in convex--non-concave and non-convex--concave min--max optimization.
For convex--non-concave min--max problems, we use a reformulation that transforms the problem into a non-concave--convex max--min optimization problem with suitably defined feasible sets and objective function. This reformulation can be viewed as an extension of the Sion's minimax theorem to the convex--non-concave setting.
We then introduce EXOTIC --- an Exact, Optimistic, Tree-based algorithm for solving the reformulated max--min problem. EXOTIC combines an iterative convex optimization solver for the inner minimization with an optimistic hierarchical tree search for the outer maximization, inspired by StroquOOL~\cite{bartlett2019simple}. Unlike StroquOOL, which assumes stochastic zero-mean noisy evaluations, EXOTIC handles deterministic, biased, and budget-dependent evaluation errors arising from finite-time solutions of the inner convex subproblems.
We establish an upper bound on its optimality gap as a function of the number of calls to the inner solver, the solver's convergence rate, and additional problem-dependent parameters. The same algorithmic framework and theoretical analysis also apply to non-convex--concave min--max optimization.

In addition, we propose a class of benchmark convex–non-concave min–max problems along with their analytical global solutions, providing a testbed for evaluating algorithms for min-max optimization. Empirically, EXOTIC outperforms gradient-based methods on this benchmark as well as on existing numerical benchmark problems from the literature. Finally, we demonstrate the utility of EXOTIC by computing security strategies in multi-player games with three or more players--a computationally challenging task that, to our knowledge, no prior method solves exactly.
\end{abstract}

\maketitle

\section{Introduction}
Min-max  optimization problems of the form 
\begin{align*}
    \min_{\mbf{x}\in \xSet}\max_{\mbf{y}\in \ySet} \obj(\mbf{x},\mbf{y})
\end{align*}
arise in diverse domains, including game theory~\cite{narahari2014game, guo2025markov}, adversarial machine learning~\cite{goodfellow2014generative, madry2018towards}, robust optimization~\cite{ben2009robust}, and control~\cite{isaacs1999differential}.
A predominant approach for solving such problems has been the use of gradient-based algorithms. Several works have established convergence guarantees of these algorithms under various assumptions on the function \(f(\cdot).\) For instance, if the function \(f(\cdot)\) is convex-concave, then \cite{korpelevich1976extragradient, nemirovski2004prox, tseng1995linear} showed asymptotic convergence to the optimal solution. 
Several works, starting from~\cite{nemirovski2004prox}, have derived non-asymptotic convergence guarantees for gradient-based methods~\cite{abernethy2019last}. For strongly-convex--strongly-concave min-max problems, lower bounds on iteration complexity have also been established~\cite{zhang2019lower, mokhtari2020unified, mokhtari2020convergence, liang2019interaction}. Near-optimal algorithms have been developed for strongly convex--strongly concave min-max optimization problems~\cite{lin2020near, wang2020improved}.
Beyond convex-concave regime, significant attention has been given to non-convex--non-concave min-max problems, which naturally arise in deep learning and adversarial training. Several works have analyzed the convergence of gradient-based algorithms in these challenging settings, including non-convex--concave \cite{rafique2022weakly, lin2020gradient, thekumparampil2019efficient, ostrovskii2021efficient}, convex--non-concave \cite{xu2023unified}, and non-convex--non-concave problems \cite{lin2020near, jin2020local, yang2020global, diakonikolas2021efficient, ostrovskii2021nonconvex, liu2021first, cai2022accelerated, zheng2023universal, cohen2025alternating}. A common metric to evaluate the performance has been to establish convergence to one of the following sets \cite{razaviyayn2020nonconvex}: 
\begin{enumerate}[labelwidth=*,align=left, widest=iii, leftmargin=*]
    \item[\textsf{(M1)}] the set of (approximate) stationary point of function \(\Phi(x) = \max_{y\in \ySet}\obj(x, y)\) \cite{lin2025two, rafique2022weakly, thekumparampil2019efficient, jin2020local};  
    \item[\textsf{(M2)}] a first order saddle point of \(\obj\) \cite{xu2023unified, nouiehed2019solving,agarwal2025framework, fiez2021local}.
\end{enumerate}
While these preceding two metrics are amenable to theoretical analysis, the resulting solutions may be far from the optimal {(see numerical evaluations in Section \ref{sec: Numerics})}.

In addition to gradient-based algorithms, there has been growing interest in gradient-free methods, motivated by settings where access to gradients is costly or unavailable. These methods typically use zeroth-order oracles to estimate gradients, which are then fed into gradient-based algorithms~\cite{maheshwari2022zeroth, beznosikov2020gradient, liu2020min, xu2020gradient}. Consequently, such methods inherit the challenges of gradient-based algorithms that focus on the optimality metrics~\textsf{(M1)}--\textsf{(M2)} discussed above.  

In this article, we develop an algorithmic framework that guarantees global minimax values for both convex--non-concave and non-convex–concave min–max optimization problems. While our primary focus is on the convex--non-concave setting, the algorithm and its accompanying analysis extend naturally to the non-convex--concave min-max optimization (see Remark~\ref{rem: NCCMM}).
The key advancements in this article are five fold:
\begin{enumerate}[labelwidth=*,align=left, widest=iii, leftmargin=*]

\item[(A)] Our \emph{first contribution} is a technical result concerning an equivalent reformulation of the convex--non-concave min-max problem as a \emph{non-concave--convex max-min} problem with modified feasible sets and an altered objective. Specifically, for the convex--non-concave problems, we show that \(
\min_{\mbf{x}\in \xSet}\max_{\mbf{y}\in \ySet} \obj(\mbf{x}, \mbf{y}) = \max_{\mbf{w}\in \setSOO} \outerObj(\mbf{w}),\)
where the set \(\setSOO\) is a set constructed from \(\ySet\) (to be defined below), {the function \(\setSOO \ni \mbf{w}\mapsto\outerObj(\mbf{w})\) is defined as \(
\outerObj(\mbf{w}) = \min_{{\mbf{x}}\in X} F({\mbf{x}}, \mbf{w}),\)
with \(X\ni x \mapsto F(\cdot, \mbf{w})\) is a convex function, for every \(\mbf{w}\in W\), constructed using \(f(\cdot)\) (see Proposition~\ref{prop: ProblemReformulate} for formal statement).} To derive this result, we first cast the original min-max problem as an equivalent \emph{convex semi-infinite program}, which is then transformed into a \emph{non-concave--convex max-min} problem using recent advances in semi-infinite programming~\cite{das2022near, ref:ParCha-23}.
{ Additionally, we provide sufficient conditions which allow us to recover stationary points and optimal solutions of the original problem and the reformulated problem (see Proposition \ref{prop:solution_and_stationary_correspondence}). In the special case where $\obj(\mbf{x},\mbf{y})$ is convex--concave, our reformulation recovers the classical minimax equality of Sion~\cite{sion1958general} {(see Proposition \ref{prop:SionMinimax})}. Thus, our result can be viewed as extending Sion's minimax theorem to convex--non-concave settings through a reformulation that modifies both the feasible set and the objective.
}

\item[(B)]
Our \emph{second contribution} is an algorithmic reduction of the reformulated minimax problem to a hierarchical optimistic search problem with an optimization-based evaluation oracle. Specifically, after reformulation, solving the original convex--non-concave minimax problem reduces to solving \(
\max_{\mbf{w}\in\setSOO} \outerObj(\mbf{w}),\)
where each evaluation of \(\outerObj(\mbf{w})\) requires solving an inner convex optimization problem. 

{To solve this problem, EXOTIC uses the hierarchical partitioning and optimistic tree-search template of StroquOOL~\cite{bartlett2019simple}. The key difference is the oracle model: StroquOOL assumes stochastic, unbiased function evaluations, whereas our evaluations are produced by a finite-time convex optimization solver and therefore have deterministic, biased, and budget-dependent errors.}

\item[(C)]
{Our \emph{third contribution} is a convergence analysis of EXOTIC under optimization-induced oracle errors; see Theorems~\ref{thm: MainConvexResult} and~\ref{thm: MainStronglyConvexResult}. 
While the outer search procedure follows the StroquOOL template \cite{bartlett2019simple}, the existing StroquOOL guarantees do not directly apply because their analysis relies on stochastic, unbiased evaluations and concentration bounds. 
In contrast, our evaluations are generated by a finite-time convex optimization solver, so the evaluation error is deterministic, biased, and explicitly dependent on the computational budget allocated to the inner solver.

We therefore extend the StroquOOL analysis for EXOTIC that couples the outer optimistic tree-search procedure with the convergence behavior of the inner solver. 
In particular, we track how the number of solver iterations allocated to each node controls the accuracy of its estimated value, and how these budget-dependent errors affect the optimistic selection and expansion rules. 
This yields optimality-gap guarantees that depend jointly on the near-optimality dimension of the outer objective and the convergence rate of the inner solver.
Specifically, when the iterative convex optimization solver converges at a polynomial rate, the optimality gap decays as \(
\tilde{\mathcal{O}}\big(n^{-1/(d + 1/\decayCpgd)}\big),\)
where \(n\) is the computational budget, \(d\) is the near-optimality dimension of the outer optimization problem (Definition~\ref{def: NearOptDim}), and \(\decayCpgd\) is the convergence-rate parameter of the inner solver; see Theorem~\ref{thm: MainConvexResult} and Corollary~\ref{cor: MainConvexResult}. 
When the inner solver converges at an exponential rate, the optimality gap improves to \(
\tilde{\mathcal{O}}(n^{-1/d});\)
see Theorem~\ref{thm: MainStronglyConvexResult} and Corollary~\ref{cor: MainStronglyConvexResult}.

These rates depend on the near-optimality dimension rather than the ambient dimension of \(\setSOO\). 
This is important because a uniform \(\epsilon\)-net discretization over a \(D\)-dimensional ambient space would require on the order of \(\epsilon^{-D}\) evaluations, whereas the tree-based procedure adapts to the smaller near-optimality dimension when the high-value region has favorable structure. 
Thus, when \(d \ll D\), EXOTIC can achieve substantially better complexity than uniform discretization.}

Unlike prior methods that target approximate stationarity in the sense of metrics \textsf{(M1)--(M2)}, our algorithmic apparatus provides, to the best of our knowledge, the first deterministic method for recovering the {globally minimax value} of convex--non-concave minimax problems under the proposed reformulation. 
Both the algorithmic apparatus and the accompanying theoretical analysis also extend to non-convex--concave minimax optimization; see Remark~\ref{rem: NCCMM}.

\item[(D)] Our \emph{fourth contribution} is a class of convex--non-concave min-max problems that have closed-form solutions (Proposition \ref{prop: HandCraftExample}). This provides a new benchmark for evaluating performance of min-max algorithms. We demonstrate that state-of-the-art gradient-based methods fail to find the {global minimax value on this benchmark and other numerical benchmarks from \texttt{SIPAMPL} database \cite{vaz_sipampl}}, while EXOTIC consistently outperforms them. 

\item[(E)] Our \emph{fifth contribution} is an application: we demonstrate that our method can compute the \emph{security value} in multi-player games with three or more players. This is an important problem in strategic decision-making providing the best worst-case cost experienced by any player in a game. This is a computationally challenging problem as it is a convex--non-concave min-max problem, and to the best of our knowledge, EXOTIC is the first algorithm that finds the security value exactly.

\end{enumerate}

\paragraph{Notations} We denote the set of positive integers by \(\N,\) and the set of real numbers by \(\R.\)
For any natural number \(n\in \N,\) we define \([n] :=\{0, 1, 2, 3,.., n\}.\) For any \(n,m\in \N,\) such that \(n\leq m,\) we define \([n:m]:=\{n, n+1, ..., m\}.\) {For any \(a,b\in \mathbb{R}\), we use the notation \(]a, b[\) to denote a open set \(\{x\in \mathbb{R}: a < x < b\}\). }

\section{Reformulation of Convex Non-concave Minimax Optimization}\label{sec: Min-maxReformulate}
Consider the following optimization problem:
\begin{equation}
\min_{\mbf{x} \in \xSet}\max_{\mbf{y} \in \ySet}f(\mbf{x}, \mbf{y})
\label{eq:min_max}    
\end{equation}
such that 
\begin{enumerate}[labelwidth=*,align=left, widest=iii, leftmargin=*]
    \item[(i)] $\xSet \subset \mathbb{R}^{\xSetDim}$ is a closed and convex set; and  $\ySet\subset \mathbb{R}^{\ySetDim}$ is a compact set;
    \item[(ii)] the function \(f:\R^{\xSetDim}\times \R^{\ySetDim}\rightarrow\R\) is continuous, and for every \(\mbf{y}\in \ySet,\) the function \(\R^{\xSetDim}\ni \mbf{x}\mapsto\obj(\mbf{x}, \mbf{y}) \in \R\) is coercive;  
    \item[(iii)] for every \(\mbf{y}\in \ySet,\) the function \(\R^{\xSetDim} \ni \mbf{x} \mapsto f(\mbf{x}, \mbf{y})\in \mathbb{R}\) is convex.
\end{enumerate}

Note that the set $\ySet$ is not required to be convex, and the mapping $\ySet\ni \mbf{y} \mapsto f(\mbf{x}, \mbf{y})\in \R$ for any $\mbf{x} \in \xSet$ need not be concave. 
It is known that \eqref{eq:min_max} admits a solution \cite{jin2020local}.
Numerically computing the solution of \eqref{eq:min_max} is challenging as the inner maximization problem is non-concave. At a high level, our approach to addressing this challenge is to reformulate the original min-max problem into a max-min problem over an appropriately defined space (see Proposition~\ref{prop: ProblemReformulate}), in which the inner minimization becomes convex. In Section~\ref{sec: Algorithm}, we present an algorithmic method that exactly solves the reformulated problem.


{\begin{proposition}\label{prop: ProblemReformulate}
The optimal value of \eqref{eq:min_max} is identical to the value of the following optimization problem:
\begin{equation}
\label{eq:CSIP_to_max_min}
\max_{\mbf{w}\in \setSOO} \left(\outerObj(\mbf{w}):= \min_{\mbf{x} \in X} F({\mbf{x}}, \mbf{w})\right), \quad \text{where} 
\end{equation}
the set \(\setSOO := \ySet^{\xSetDim+1}\); for any \(\mbf{w} = [\mbf{y}_1^\top, \mbf{y}_2^\top, \dots, \mbf{y}_{\xSetDim+1}^\top]^\top \in W, \) the mapping 
\begin{align}\label{eq: g_equation}
   X\ni \mbf{x}\mapsto F(\mbf{x}, \mbf{w}) := \max_{i\in [d_x+1]}f(\mbf{x},\mbf{y}_i).
\end{align}
\end{proposition}
}
A proof of Proposition \ref{prop: ProblemReformulate} is provided in Appendix \ref{app: ProofMinmaxToMaxmin}.  To derive this result, we first cast the original min-max problem as an equivalent \emph{convex semi-infinite program}, which is then transformed into a \emph{non-concave--convex max-min} problem using recent advances in semi-infinite programming~\cite{das2022near, ref:ParCha-23}. 

\begin{remark}
The problem in \eqref{eq:CSIP_to_max_min} is a non-concave--convex max-min optimization. 
{Indeed, for any \(\mbf{w}=(\mbf{y}_i)_{i\in[d_x+1]}\in W,\) the function \(X\ni\mbf{x}\mapsto F(\mbf{x},\mbf{w})\) is a point-wise maximum of convex functions, and hence convex. } 
\end{remark}

{{}
\begin{remark}
Proposition~\ref{prop: ProblemReformulate} establishes equivalence of the optimal values of \eqref{eq:min_max} and \eqref{eq:CSIP_to_max_min}. This equivalence, however, should not be interpreted as a one-to-one correspondence between the optimizer sets, or between the stationary points, of the two problems. To see the issue, let \(\mbf{w}^\ast\in \arg\max_{\mbf{w}\in W}\min_{\mbf{x}\in X}F(\mbf{x},\mbf{w})\), and let \(\mbf{x}^\ast\in \arg\min_{\mbf{x}\in X}F(\mbf{x},\mbf{w}^\ast)\). Although \(\mbf{w}^\ast\) is globally optimal for the reformulated problem, the induced minimizer \(\mbf{x}^\ast\) need not, in general, belong to \(\arg\min_{\mbf{x}\in X}\max_{\mbf{y}\in Y}f(\mbf{x},\mbf{y})\).\footnote{{{}For example, let \(X=[0,1]\), \(Y=[0,1]\), and \(f(\mbf{x},\mbf{y})=\mbf{xy}\). Then \(\max_{\mbf{y}\in Y}f(\mbf{x}, \mbf{y})=\mbf{x}\), so the original minimax problem has the unique solution \(\bar{\mbf{x}}=0\). In this example, \(F(\mbf{x},\mbf{w})=\mbf{x}\max\{\mbf{y}_1,\mbf{y}_2\}\). Hence \(\min_{\mbf{x}\in X}F(\mbf{x, w})=0\) for every \(\mbf{w}\in Y^2\). Pick \(\mbf{w}^\ast=(0,0)\) and \(\mbf{x}^\ast=1/2\). But note that \(\mbf{x}^\ast=1/2\) is not optimal for the original minimax problem.}} Thus, optimality of the tuple \(\mbf{w}^\ast\) does not by itself guarantee recovery of a globally optimal solution of the original minimax problem.
The following result clarifies when such a recovery is valid. It also states the corresponding condition under which stationarity of the induced finite-max problem is equivalent to stationarity of the original minimax objective.
\end{remark}
\begin{proposition}
\label{prop:solution_and_stationary_correspondence}
Consider arbitrary \(\mbf{w}^\ast\in \arg\max_{\mbf{w}\in W} \min_{\mbf{x}\in X} F(\mbf{x},\mbf{w})\) and \(
    \mbf{x}^\ast\in
    \arg\min_{\mbf{x}\in X}F(\mbf{x},\mbf{w}^\ast).\) Define \(
    \Phi(\mbf{x}) := \max_{\mbf{y}\in Y} f(\mbf{x},\mbf{y})\). 
Then \(\mbf{x}^\ast\) is a global solution of the original minimax problem if either of the following conditions holds:
\begin{itemize}
    \item[(i)] \(\Phi(\cdot)\) and \(F(\cdot,\mbf{w}^\ast)\) have a common unique minimizer;\footnote{ in particular, this holds if both objectives are strictly convex on \(X\).}
    \item[(ii)] the finite tuple \(\mbf{w}^\ast\) is exact at \(\mbf{x}^\ast\), namely \(F(\mbf{x}^\ast,\mbf{w}^\ast)=\Phi(\mbf{x}^\ast).\)
\end{itemize}

Furthermore, suppose \(f(\cdot,\mbf{y})\) is continuously differentiable for every \(\mbf{y}\in Y\). If the finite tuple \(\mbf{w}^\ast\) is first-order exact at \(\mbf{x}^\ast\), namely \(    \partial_{\mbf{x}}F(\mbf{x}^\ast,\mbf{w}^\ast)
    =
    \partial \Phi(\mbf{x}^\ast),\)
then \(\mbf{x}^\ast\) is first-order stationary for \(
    \min_{\mbf{x}\in X}F(\mbf{x},\mbf{w}^\ast)\)
if and only if \(\mbf{x}^\ast\) is first-order stationary for \(    \min_{\mbf{x}\in X}\Phi(\mbf{x}).\)
\end{proposition}
A proof of Proposition \ref{prop:solution_and_stationary_correspondence} is provided in Appendix \ref{app: StationaryPoint}.
}

{{}
\begin{remark}\label{rem: SionGeneralization} Proposition~\ref{prop: ProblemReformulate} 
guarantees that for a convex--non-concave function \(
\min_{\mbf{x}\in X}\max_{\mbf{y}\in Y}f(\mbf{x}, \mbf{y}) = \max_{\mbf{w}\in Y^{d_x+1}}\min_{\mbf{x}\in X}F(\mbf{x},\mbf{w}).\) That is, we can convert a min-max problem to a max-min problem where the domain of maximization and the objective function are appropriately changed. This 
can be viewed as a generalization of Sion’s minimax theorem~\cite{sion1958general}. Recall that, Sion’s minimax theorem---a foundational result in game theory---guarantees that  
\(
\min_{\mbf{x} \in \xSet} \max_{\mbf{y} \in \ySet} f(\mbf{x}, \mbf{y}) = \max_{\mbf{y} \in \ySet} \min_{\mbf{x} \in \xSet} f(\mbf{x}, \mbf{y})
\)
for any continuous function \(f:\xSet\times\ySet\rightarrow \R\) that is convex--concave.  
Indeed, when \(f\) is convex-concave, we show that using Proposition \ref{prop: ProblemReformulate}, we can recover Sion's minimax theorem as shown below. 
\end{remark}
\begin{proposition}\label{prop:SionMinimax}
Let $\xSet \subset \mathbb{R}^{\xSetDim}$ and $\ySet \subset \mathbb{R}^{d_y}$ be nonempty, convex, and compact. Let $f : \xSet \times \ySet \to \mathbb{R}$ be continuous, convex in $\mbf{x}$ for every fixed $\mbf{y} \in \ySet$, and concave in $\mbf{y}$ for every fixed $\mbf{x} \in \xSet$. Then, using Proposition 2.1, we can show that  
\[
\min_{\mbf{x}\in \xSet}\max_{\mbf{y}\in \ySet} f(\mbf{x},\mbf{y})
=
\max_{\mbf{y}\in \ySet}\min_{\mbf{x}\in \xSet} f(\mbf{x},\mbf{y}).
\]
\end{proposition}
A proof of Proposition \ref{prop:SionMinimax} is provided in Appendix \ref{app: SionMinimax}.
}

\section{Algorithm}\label{sec: Algorithm}
{In this section, we introduce our approach to solve~\eqref{eq:min_max}. Recall from Proposition~\ref{prop: ProblemReformulate} that~\eqref{eq:min_max} can be reformulated as
\begin{align}\label{eq: SolveReformulation} 
\max_{\mbf{w} \in \setSOO} \outerObj(\mbf{w}), \quad \text{where} \quad 
\setSOO \ni \mbf{w} \mapsto \outerObj(\mbf{w}) := \min_{{\mbf{x}} \in X} F(\mbf{x},\mbf{w}).
\end{align}

Our algorithmic approach builds on the StroquOOL algorithm of \cite{bartlett2019simple}, a hierarchical optimistic optimization method designed for global maximization of non-convex functions under stochastic oracle feedback with zero-mean noise. 
Our approach preserves the core structural components of StroquOOL—namely hierarchical partitioning of the search space, optimistic node selection, and progressive budget allocation across depths. 

The key distinction lies in the oracle model. In our setting, the objective \(G(\mbf{w})\) is defined as \(
G(\mbf{w}) = \min_{\mbf{x}\in X} F(\mbf{x}, \mbf{w}),\) for which we do not have stochastic oracle. Instead, for any \(\mbf{w}\), we obtain an estimate of \(G(\mbf{w})\) by running an optimization solver \(\OPT\) for a finite number of iterations. Consequently, each evaluation is deterministic, biased, and depends explicitly on the computational budget and initialization. 

As a result, the distinction from StroquOOL is not in the tree-based search structure, but in the oracle model and its analysis. In StroquOOL, the objective evaluations are corrupted by stochastic, zero-mean noise, and estimation error is controlled via concentration arguments. In contrast, in our setting each evaluation is obtained via a finite-time optimization procedure, resulting in deterministic, biased, and iteration-dependent errors. Consequently, classical concentration-based arguments do not apply, and the estimation accuracy must instead be characterized through convergence guarantees of the optimization solver.
}

{{}\subsection{Optimization Solver \OPT}
As discussed above, we use an iterative convex optimization solver \OPT\ to evaluate \(\outerObj(\mbf{w})\), for any \(\mbf{w}\in W.\) Typically, these solvers only converge to optimality asymptotically. Therefore, after finitely many iterations of \OPT, we only have an approximate evaluation of \(\outerObj(\mbf{w})\). 
More concretely, 
the solver is denoted by \(\OPT: \setSOO \times \xSet \times \N \rightarrow \xSet\), which outputs an approximate solution \(\OPT(\mbf{w},{\mbf{z}}, \solIter)\) after \(\solIter\) iterations, initialized at \({\mbf{z}} \in \xSet\). We define the corresponding estimated value as \(
\underline{G}(\mbf{w}, {\mbf{z}}, \solIter) := F(\OPT(\mbf{w}, {\mbf{z}}, \solIter), \mbf{w}).\)

We characterize the accuracy of this oracle through the following assumptions, which will be used in the convergence analysis of the overall algorithm.}

\begin{assumption}\label{assm: OPT_Convergence_Convex}
The optimization procedure \(\textsf{OPT}\) is such that for every \(\mbf{w} \in \setSOO\), {{}\({\mbf{z}} \in X\), we have \(\OPT(\mbf{w}, {\mbf{z}}, \solIter)\in X\)} and
\[
    \underline{G}(\mbf{w}, {\mbf{z}}, \solIter) \leq G(\mbf{w}) + \Cpgd \, \solIter^{-\decayCpgd}, \quad \forall \solIter \in \mathbb{N},
\]
where \(\Cpgd > 0\) and \(\decayCpgd > 0\) characterize the sublinear convergence rate.
\end{assumption}

\begin{assumption}\label{assm: OPT_Convergence_StronglyConvex}
The optimization procedure \(\textsf{OPT}\) is such that for every \(\mbf{w} \in \setSOO\), {{}\({\mbf{z}} \in X\), we have \(\OPT(\mbf{w}, {\mbf{z}}, \solIter)\in X\)} and
\[
    \underline{G}(\mbf{w}, {\mbf{z}}, \solIter) \leq G(\mbf{w}) + \SCpgd \cdot \decaySCpgd^{\solIter}, \quad \forall \solIter \in \mathbb{N},
\]
where \(\SCpgd > 0\) and \(\decaySCpgd \in (0,1)\) characterize the linear convergence rate.
\end{assumption}

{{}In Appendix \ref{app:SufficientConditionsAssm12}, we provide sufficient conditions for Assumptions \ref{assm: OPT_Convergence_Convex}-\ref{assm: OPT_Convergence_StronglyConvex}.}

{{}\subsection{Algorithmic Apparatus}
The algorithm follows the hierarchical partitioning and optimistic node selection strategy of StroquOOL \cite{bartlett2019simple}, but adapts it to use the optimization-based oracle defined above, instead of stochastic oracle. We provide algorithmic description below for completeness.

We construct a hierarchical partition of \(\setSOO\) parameterized by \(\kTree\). At depth \(h\), the domain is partitioned into \(\kTree^h\) disjoint subsets \(\{\setSOO_{h,i}\}_{i\in[\kTree^h]}\), inducing a \(\kTree\)-ary tree \(\tree\), where each node \((h,i)\) corresponds to a subset \(\setSOO_{h,i}\). 
Each node \((h,i)\) is associated with: \textit{(i)} a representative point \(\varSOO_{h,i} \in \setSOO_{h,i}\); \textit{(ii)} an approximate value \(\estSOO_{h,i}\); \textit{(iii)} an initializer \(\optSOO_{h,i}\) for the solver \(\OPT\); \textit{(iv)} number of iterations of \OPT \ executed on that node, denoted by \(\num_{h,i}\).

The approximate value \(\estSOO_{h,i}\) is then computed as 
\[
\estSOO_{h,i} = F\big(\OPT(\varSOO_{h,i}, \optSOO_{h,i}, \num_{h,i}), \varSOO_{h,i}\big).
\]

\medskip
\noindent\textbf{Initialization.}
We initialize the tree with the root node and its \(\kTree\) children at depth \(h=1\). Each node \((1,i)\) is evaluated using \(\OPT\) for \(\hmax\) iterations.

\medskip
\noindent\textbf{Tree expansion.}
For each depth \(h \in [\hmax]\), and for each \(m \in [\lfloor \hmax/h \rfloor]\), we select a leaf node \((h,\bar{i})\) satisfying \textit{(a)} \(\num_{h,\bar{i}} \ge \lfloor \hmax/(hm)\rfloor\), and \textit{(b)} it maximizes \(\estSOO_{h,i}\) over all leaf nodes at that depth. We then expand this node by adding \(K\) child nodes and evaluate all child nodes using \OPT\ for \(\lfloor \hmax/(hm)\rfloor\) iterations.

\medskip
\noindent\textbf{Re-evaluation.}
Finally, a re-evaluation phase is performed because in the tree-expansion phase different nodes are evaluated with different levels of accuracy. For each 
\(p \in [\lfloor \log_2(h_{\max}) \rfloor]\), we identify the node with the largest value\footnote{Ties are broken arbitrarily.} of \(\estSOO_{h,i}\) among all nodes that have been evaluated by \(\OPT\) for at least \(2^p\) iterations. This gives at most \(\lfloor \log_2(h_{\max}) \rfloor\) candidate nodes. Each candidate node is then re-evaluated by running \(\OPT\) for a fixed budget of \(\lfloor h_{\max}/2 \rfloor\) iterations. The algorithm returns the node with the largest estimated value after the re-evaluation phase.
The complete procedure is given in Algorithm~\ref{alg: Algorithm_ComputingSol}.}

\begin{algorithm}[h!]
\caption{{\textbf{EXOTIC:} \textbf{EX}act, \textbf{O}ptimistic, \textbf{T}ree-based Algor\textbf{I}thm for \textbf{C}onvex--Non-concave (and Non-convex--Concave) {Minimax} Optimization }}
\begin{algorithmic}[1]
\STATE \textbf{Initialize:} 
\begin{itemize}
  \item The total number of iterations of the optimization solver to \(\iter\) and the maximum depth of the tree \(\hmax.\)
  \item For any \(i\in [\kTree^h], h\in [\hmax],\) set representative points \(\varSOO_{h,i}\in \setSOO_{h,i}\) and \(\optSOO_{h,i} \in \setmap(\varSOO_{h,i})\)
   \item Tree \(\tree\) with root node \(\{(0,1) = \varSOO_{0,1}, \estSOO_{0,1} \leftarrow 0, \optSOO_{0,1}\}\). 
  \item {{}Append the nodes \(\{(1,i) = (\varSOO_{1,i}, \estSOO_{1,i} \leftarrow F(\OPT(\varSOO_{1,i}, \optSOO_{h,i}, \hmax), \varSOO_{1,i}), \optSOO_{h,i})\}_{i \in [\kTree]}\) as children of \((0,1)\) in \(\tree\).}
  \item Set \(\{\num_{1,i} \leftarrow \hmax\}_{i \in [\kTree]}\).
\end{itemize}

{{}\%\textsf{Tree-expansion phase}}
\FOR{\(h = 1\) to \(\hmax\)}
    \FOR{\(m = 1\) to \(\lfloor \hmax / h \rfloor\)}
        \STATE Select a leaf node \((h, \bar{i})\) such that  \(
            \bar{i} \in \underset{(h,i)\in \tree_{\text{leaf}} ~ \text{s.t.} ~  \num_{h,i} \geq \lfloor \hmax/(hm)\rfloor} {\arg\max} ~ \estSOO_{h, i},
        \)
        where \(\tree_{\text{leaf}}\) denotes the set of leaf nodes of \(\tree\)
        \STATE Update \(\tree\) by adding \(\kTree\) child nodes \(\{ (h+1, j) = (\varSOO_{h+1, j}, \estSOO_{h+1, j}\leftarrow 0, \optSOO_{h+1, j})\}_{j\in [\kTree]}\) to the node \((h,\bar{i})\)
        \FOR{\(j = 1\) to \(j = \kTree\)}  
            \STATE {{}Set \(\estSOO_{h+1, j} \leftarrow F(\OPT(\varSOO_{h+1,j},  \optSOO_{h+1,j}, \lfloor \hmax/(hm) \rfloor),\varSOO_{h+1,j} )\)}
            \STATE Set \(\num_{h+1,j} \leftarrow \lfloor \hmax/(hm) \rfloor \)
        \ENDFOR
    \ENDFOR
\ENDFOR

{{}\%\textsf{Re-evaluation phase}}

\FOR{\(p = 0\) to \(\lfloor \log_2(\hmax) \rfloor\)}
    \STATE \((h_p, i_p) \leftarrow \underset{(h,i)\in \tree :~ \num_{h,i} \geq 2^p}{\arg\max} \estSOO_{h,i},\) 
    \STATE {{}\(\estSOO_{h_p, i_p} \leftarrow F(\OPT(\varSOO_{h_p,i_p},  \optSOO_{h_p,i_p}, \lfloor \hmax/2 \rfloor), \varSOO_{h_p,i_p})\)}
\ENDFOR
\STATE Set \(\pOut \leftarrow \underset{p\in [\lfloor\log_2(\hmax)\rfloor]}{\arg\max} ~ \estSOO_{h_p, i_p}\), \(\returnSol \leftarrow \varSOO_{h_{\pOut}, i_{\pOut}},\) and \(\returnVal \leftarrow \estSOO_{h_{\pOut}, i_{\pOut}}\)
\RETURN \(\returnVal, \returnSol\)
\end{algorithmic}\label{alg: Algorithm_ComputingSol}
\end{algorithm}

{{}\subsection{Analysis} 
We now evaluate the performance of Algorithm~\ref{alg: Algorithm_ComputingSol} under the convergence assumptions on \(\OPT\). 
Our algorithmic pipeline follows the hierarchical optimistic search structure of StroquOOL~\cite{bartlett2019simple}; however, the oracle model is fundamentally different. StroquOOL assumes stochastic evaluations with zero-mean noise, so that the error after \(s\) evaluations can be controlled using high-probability concentration bounds, leading to the canonical \(\mathcal{O}(s^{-1/2})\) decay. In contrast, in our setting each evaluation is produced by a finite-time optimization procedure. The resulting error is deterministic, biased, and explicitly dependent on the optimization budget, as characterized in Assumption~\ref{assm: OPT_Convergence_Convex}. Consequently, although our proof strategy builds on the StroquOOL framework, the convergence analysis must be re-derived to account for this budget-dependent deterministic error model.

Furthermore, when the finite-time optimization oracle admits exponential error decay, as in Assumption~\ref{assm: OPT_Convergence_StronglyConvex}, we obtain sharper convergence rates.}

We start by demonstrating that for any \(\iter \in \N\), we can select the maximum tree depth \(\hmax\) in Algorithm~\ref{alg: Algorithm_ComputingSol} to ensure that the total number of iterations of the optimization solver \(\OPT\) is bounded by \(\iter\).

\begin{proposition}\label{prop: Num_Iterations}
For any positive integer \(n,\) if we set
\(
\hmax = \left\lfloor \frac{2\iter}{5\kTree(1 + \log(\iter))^2} \right\rfloor,
\)
then the total number of iterations of the optimization solver \(\OPT\) is bounded above by \(\iter\).
\end{proposition}

A proof of Proposition~\ref{prop: Num_Iterations} is provided in the Appendix \ref{app: NumIterations}. 

The metric to evaluate the performance of the algorithm is the \emph{optimality gap} of the solution returned by the algorithm as a function of total number of iterations \(\iter\) of \(\OPT\): \(
    \regret(\iter) := \max_{\mbf{w} \in \setSOO} \outerObj(\mbf{w}) - \returnVal,\)
where \(\returnVal\) denotes value returned by the algorithm after \(\iter\) solver iterations (see \textsf{line 16} in Algorithm \ref{alg: Algorithm_ComputingSol}). 
Naturally, in order to study the theoretical convergence properties of Algorithm \ref{alg: Algorithm_ComputingSol}, one needs certain regularity conditions on the function \(\outerObj(\cdot).\) 
\begin{assumption}\label{assm: Func_Part_Assm}
    For any \(\mbf{w}^\ast \in \arg\max_{\mbf{w}\in \setSOO}\outerObj(\mbf{w}),\) there exists \(\constDecay > 0\) and \(\rateDecay\in\, ]0,1[\) such that  \(
        \outerObj(\mbf{w}) \geq \outerObj(\mbf{w}^\ast) - \constDecay \rateDecay^h\) {for all}  \((h, \mbf{w}) \in \N \times \setSOO_{h,i^\ast_h}\), 
    where \(i^\ast_h\) is the index of the partition at depth \(h\) containing \(\mbf{w}^\ast\). 
\end{assumption}
Assumption~\ref{assm: Func_Part_Assm} ensures that the sub-optimality of the function \(\outerObj\) within partitions containing the optimal solution decays geometrically with the size of the partition. In Appendix~\ref{app: SufficientConditions}, we provide sufficient conditions on the minimax optimization problem~\eqref{eq:min_max} that ensure Assumptions~\ref{assm: Func_Part_Assm} is satisfied (see Proposition~\ref{prop: SufficientConditins}).

Next, we introduce 
near-optimality dimension, as defined below: 
\begin{definition}[Near-optimality Dimension \cite{bartlett2019simple}]\label{def: NearOptDim}
    For any \(\constDecay > 0\), \(C > 1\), and \(\rateDecay \in ~]0,1[\), we define \(d(\constDecay, \rateDecay, C)\) to be the \emph{near-optimality dimension} of the function \(\outerObj\) with respect to the partitioning \(\{\setSOO_{h,i}\}_{h\in[\hmax], i\in [\kTree^h]}\) if
\begin{align*}
    d(\constDecay, \rateDecay, C) := \inf\left\{ d' \in \mathbb{R}_+ \,\Big|\, \text{for every}~ h \in [\hmax], \ \mathcal{N}_h(2\constDecay \rateDecay^h) \leq C \rateDecay^{-d'h} \right\},
\end{align*}
where for any \(\epsilon\geq 0, h\in [\hmax]\) and \(\mathcal{N}_h(\epsilon) := \{ i \in [\kTree^h] \mid 
    \max_{\mbf{w} \in \setSOO_{h,i}} \outerObj(\mbf{w}) \geq \outerObj(\mbf{w}^\ast) - \epsilon\}.
\)
\end{definition}

We are now ready to state the rate of convergence of the optimality gap. Naturally, this rate depends on the convergence rate of the optimization solver~\(\OPT\). We first provide the convergence rate of the optimality gap when the solver~\(\OPT\) satisfies Assumption~\ref{assm: OPT_Convergence_Convex}.
\begin{theorem}\label{thm: MainConvexResult}
Suppose that Assumptions~\ref{assm: OPT_Convergence_Convex} and \ref{assm: Func_Part_Assm} hold with \(d\) as the near -optimality dimension of \(\outerObj.\) Let \(\lamW\) be the Lambert W-function (refer Appendix \ref{app: LambertW}).  
\begin{enumerate}[labelwidth=*,align=left, widest=iii, leftmargin=*]
    \item[(i)] If the condition 
    \begin{align}\label{eq: ConditionTheorem}
        \constDecay \exp\left({-\frac{1}{\left(d+\frac{1}{\decayCpgd}\right)}}\mathcal{W}\left( \frac{\left\lfloor \frac{2n}{5\kTree(1+\log(n))^2}\right\rfloor\left(d+\frac{1}{\decayCpgd}\right)\log(\frac{1}{\rateDecay})}{4C\Cpgd^{1/\decayCpgd}\constDecay^{-1/\decayCpgd} } \right)\right) \leq 2^{\decayCpgd}\Cpgd
    \end{align}
    holds,
    then the optimality gap of Algorithm~\ref{alg: Algorithm_ComputingSol}, as a function of the number of iterations of the optimization solver \(\OPT\), satisfies the bound 
{\small \begin{align}\label{eq: GapCondition_satisfied}
        \regret(\iter) \leq &\frac{\Cpgd}{\left\lfloor \frac{n}{5\kTree(1+\log(n))^2}\right\rfloor^{\decayCpgd}} + \constDecay \exp\left({-\frac{\mathcal{W}\left( \frac{\left\lfloor \frac{2n}{5\kTree(1+\log(n))^2}\right\rfloor\left(d+\frac{1}{\decayCpgd}\right)\log(\frac{1}{\rateDecay})}{4C\Cpgd^{1/\decayCpgd}\constDecay^{-1/\decayCpgd} } \right)}{\left(d+\frac{1}{\decayCpgd}\right)}}\right).
    \end{align}}

    \item[(ii)] If condition~\eqref{eq: ConditionTheorem} is not satisfied, then the optimality gap of Algorithm~\ref{alg: Algorithm_ComputingSol}, as a function of the number of iterations of the optimization solver \(\OPT\), satisfies the bound 
\begin{align}\label{eq: ConditionnotSatisfied}
        \regret(\iter) \leq \begin{cases}2\constDecay\exp\left(-\frac{1}{d}\mathcal{W}\left(\frac{d\log(1/\rateDecay)\left\lfloor \frac{2n}{5\kTree(1+\log(n))^2}\right\rfloor}{2C}\right)\right), & \text{if} \ d > 0, \\ 
       2\constDecay \rho^{\frac{1}{2C}\left\lfloor \frac{2n}{5\kTree(1+\log(n))^2}\right\rfloor}, &\text{if} \ d = 0.
        \end{cases}
    \end{align}
\end{enumerate}
\end{theorem}
A proof of Theorem \ref{thm: MainConvexResult} is provided in Appendix \ref{ssec: FirstThm}. 
\begin{remark}\label{rem: ConvergenceConvex}
Since \(\lamW(\cdot)\) is monotonically increasing on the positive real axis (see Lemma \ref{lem: W_properties}), the bound on the optimality gap in \eqref{eq: GapCondition_satisfied}--\eqref{eq: ConditionnotSatisfied} vanishes as \(\iter\) increases. Combined with Proposition~\ref{prop: ProblemReformulate}, Theorem~\ref{thm: MainConvexResult} ensures that Algorithm~\ref{alg: Algorithm_ComputingSol} converges to the exact solution of the convex--non-concave min-max optimization problem. {{} Furthermore, we obtain exponentially fast convergence when \(d=0\).}
\end{remark}
In the high-\(\iter\) regime, leveraging the properties of the Lambert \(W\)-function (see Lemma~\ref{lem: W_properties}-(iii)), the bounds presented in Theorem~\ref{thm: MainConvexResult} can be refined, as stated in the following result. 
\begin{corollary}\label{cor: MainConvexResult}
Suppose that Assumptions~\ref{assm: OPT_Convergence_Convex} and \ref{assm: Func_Part_Assm}  hold with \(d > 0\) as the near-optimality dimension of \(\outerObj.\)  
\begin{enumerate}[labelwidth=*,align=left, widest=iii, leftmargin=*]
    \item[(i)] If the condition~\eqref{eq: ConditionTheorem} is satisfied and
    \begin{align}
        \label{eq: Large_n_convex_condition_satisfied}\frac{\left\lfloor \frac{2n}{5\kTree(1+\log(n))^2}\right\rfloor \left(d+\frac{1}{\decayCpgd}\right)\log(\frac{1}{\rateDecay})}{4C\Cpgd^{1/\decayCpgd}\constDecay^{-1/\decayCpgd}} > e,
    \end{align}
    then the optimality gap of Algorithm~\ref{alg: Algorithm_ComputingSol}, as a function of the number of iterations of the optimization solver \(\OPT\), satisfies 
    \begin{align}
        \regret(\iter) \leq \frac{\Cpgd}{\left\lfloor \frac{n}{5\kTree(1+\log(n))^2}\right\rfloor^{\decayCpgd}} + \constDecay\left(\frac{\frac{\left\lfloor \frac{2n}{5\kTree(1+\log(n))^2}\right\rfloor\left(d+\frac{1}{\decayCpgd}\right)\log(\frac{1}{\rateDecay})}{4C\Cpgd^{1/\decayCpgd}\constDecay^{-1/\decayCpgd}}}{\log\left(\frac{\left\lfloor \frac{2n}{5\kTree(1+\log(n))^2}\right\rfloor \left(d+\frac{1}{\decayCpgd}\right)\log(\frac{1}{\rateDecay})}{4C\Cpgd^{1/\decayCpgd}\constDecay^{-1/\decayCpgd}} \right)}\right)^{-\frac{1}{d+\frac{1}{\decayCpgd}}}.
    \end{align}

    \item[(ii)] If the condition~\eqref{eq: ConditionTheorem} is not satisfied but 
    \begin{align}\label{eq: Large_n_convex_condition_not_satisfied}
        \frac{d\log(1/\rateDecay)\left\lfloor \frac{2n}{5\kTree(1+\log(n))^2}\right\rfloor}{2C} > e,
    \end{align}
    then the optimality gap of Algorithm~\ref{alg: Algorithm_ComputingSol}, as a function of the number of iterations of the optimization solver \(\OPT\), satisfies 
    \begin{align}\label{eq: CorollaryRateConditionnotSatisfied}
        \regret(\iter) \leq 2\constDecay\left(\frac{\frac{d\log(1/\rateDecay)\left\lfloor \frac{2n}{5\kTree(1+\log(n))^2}\right\rfloor}{2C}}{\log\left(\frac{d\log(1/\rateDecay)\left\lfloor \frac{2n}{5\kTree(1+\log(n))^2}\right\rfloor}{2C}\right)}\right)^{-1/d}.
    \end{align}
    
\end{enumerate}
\end{corollary}

\begin{remark}\label{rem: CorollaryConvex}
    Condition~\eqref{eq: ConditionTheorem}, intuitively, holds when the optimization solver \(\OPT\) has poor initialization or a small step size (e.g., when \(\Cpgd\) is large). Such poor initialization naturally leads to a larger bound on the optimality gap. 
Indeed, when \(\iter\) is sufficiently large, conditions~\eqref{eq: Large_n_convex_condition_satisfied} and~\eqref{eq: Large_n_convex_condition_not_satisfied} are satisfied and Corollary~\ref{cor: MainConvexResult} guarantees that the optimality gap decays at a rate of \(\tilde{\mathcal{O}}(n^{-1/(d + 1/\decayCpgd)})\) if condition~\eqref{eq: ConditionTheorem} holds, and at a faster rate of \(\tilde{\mathcal{O}}(n^{-1/d})\) otherwise.
\end{remark}

Next, we show that the rate of convergence of optimality gap can be improved if we allow \(\OPT\) to converge exponentially fast (i.e., Assumption \ref{assm: OPT_Convergence_StronglyConvex} holds). 

\begin{theorem}\label{thm: MainStronglyConvexResult}
Suppose that Assumptions~\ref{assm: OPT_Convergence_StronglyConvex} and ~\ref{assm: Func_Part_Assm} hold with \(d\) as the near-optimality dimension of \(\outerObj.\) Let \(\lamW\) be the Lambert W-function (refer Appendix \ref{app: LambertW}). 
\begin{enumerate}[labelwidth=*,align=left, widest=iii, leftmargin=*]
    \item[(i)] If the condition
    \begin{align}\label{eq: ConditionTheoremSC}
       \exp\left(-\frac{2}{d}\mathcal{W}\left( \frac{d}{2}\log\left(\frac{1}{\rateDecay}\right)\sqrt{\frac{\left\lfloor \frac{2n}{5\kTree(1+\log(n))^2}\right\rfloor}{4C}} \right)\right) \leq \gamma_{\textsf{sc}}
    \end{align}
    holds, then the optimality gap of Algorithm~\ref{alg: Algorithm_ComputingSol}, as a function of the number of iterations of the optimization solver \(\OPT\), satisfies the bound 
{\small\begin{align}\label{eq: GapCondition_satisfiedSC}
        \regret(\iter) \leq \SCpgd\decaySCpgd^{\left\lfloor \frac{n}{5\kTree(1+\log(n))^2}\right\rfloor}+ \SCpgd \exp\left(-\frac{2}{d}\mathcal{W}\left( \frac{d}{2}\log\left(\frac{1}{\rateDecay}\right)\sqrt{\frac{\left\lfloor \frac{2n}{5\kTree(1+\log(n))^2}\right\rfloor}{4C}} \right)\right).
    \end{align}}

    \item[(ii)] If condition~\eqref{eq: ConditionTheoremSC} is not satisfied, then the optimality gap of Algorithm~\ref{alg: Algorithm_ComputingSol}, as a function of the number of iterations  of the optimization solver \(\OPT\), satisfies the bound 
\begin{align}\label{eq: ConditionnotSatisfiedSC}
        \regret(\iter) \leq \begin{cases}
        2\SCpgd\exp\left(-\frac{1}{d}\lamW\left(\frac{d\log(1/\rateDecay)\left\lfloor \frac{2n}{5\kTree(1+\log(n))^2}\right\rfloor}{2C}\right)\right), &\text{if} \ d > 0, 
        \\ 
        2\SCpgd \rho^{\frac{1}{2C}\left\lfloor \frac{2n}{5\kTree(1+\log(n))^2}\right\rfloor}, & \text{if} \ d= 0.
        \end{cases}
    \end{align}
\end{enumerate}
\end{theorem}
A proof of Theorem \ref{thm: MainStronglyConvexResult} is provided in Appendix \ref{ssec: SecondThm}. 

In the high-\(\iter\) regime, leveraging the properties of the Lambert \(W\)-function (see Lemma \ref{lem: W_properties}-(iii)), the bounds presented in Theorem~\ref{thm: MainConvexResult} can be refined, as stated in the following result. 

\begin{corollary}\label{cor: MainStronglyConvexResult}
Suppose that Assumptions~\ref{assm: OPT_Convergence_StronglyConvex} and  ~\ref{assm: Func_Part_Assm} hold with \(d\) as the near-optimality dimension of \(\outerObj.\)
\begin{enumerate}[labelwidth=*,align=left, widest=iii, leftmargin=*]
    \item[(i)] If condition~\eqref{eq: ConditionTheoremSC} is satisfied but 
    \begin{align}\label{eq: Large_n_strongly_convex_condition_satisfied}
     \frac{d}{2}\log\left(\frac{1}{\rateDecay}\right)\sqrt{\frac{\left\lfloor \frac{2n}{5\kTree(1+\log(n))^2}\right\rfloor}{4C}} > e,
\end{align}
    then the optimality gap of Algorithm~\ref{alg: Algorithm_ComputingSol}, as a function of the number of iterations of the optimization solver \(\OPT\), is bounded as:
    \begin{align}
    \regret(\iter) \leq \SCpgd\decaySCpgd^{\left\lfloor \frac{n}{5\kTree(1+\log(n))^2}\right\rfloor} + \left( 
   \frac{ \frac{d}{2}\log\left(\frac{1}{\rateDecay}\right)\sqrt{\frac{\left\lfloor \frac{2n}{5\kTree(1+\log(n))^2}\right\rfloor}{4C}}}{\log\left( \frac{d}{2}\log\left(\frac{1}{\rateDecay}\right)\sqrt{\frac{\left\lfloor \frac{2n}{5\kTree(1+\log(n))^2}\right\rfloor}{4C}}\right)} \right)^{-2/d}.
\end{align}
    \item[(ii)] If condition~\eqref{eq: ConditionTheoremSC} is not satisfied but \eqref{eq: Large_n_convex_condition_not_satisfied} holds,
    then the optimality gap of Algorithm~\ref{alg: Algorithm_ComputingSol}, as a function of the number of iterations of the optimization solver \(\OPT\), is bounded as:
    \begin{align}\label{eq: CorollaryRateConditionnotSatisfiedSC}
        \regret(\iter) \leq 2\constDecay\left(\frac{\frac{d\log(1/\rateDecay)\left\lfloor \frac{2n}{5\kTree(1+\log(n))^2}\right\rfloor}{2C}}{\log\left(\frac{d\log(1/\rateDecay)\left\lfloor \frac{2n}{5\kTree(1+\log(n))^2}\right\rfloor}{2C}\right)}\right)^{-1/d}.
    \end{align}
    \end{enumerate}
\end{corollary}

    The bounds presented in Theorem~\ref{thm: MainStronglyConvexResult} (respectively, Corollary~\ref{cor: MainStronglyConvexResult}) are tighter than those in Theorem~\ref{thm: MainConvexResult} (respectively, Corollary~\ref{cor: MainConvexResult}). This improvement is possible because Assumption~\ref{assm: OPT_Convergence_StronglyConvex} imposes a stronger convergence requirement on the optimization solver \(\OPT\) than Assumption~\ref{assm: OPT_Convergence_Convex}. 

{{}\begin{remark}
The above rate should be contrasted with a uniform $\epsilon$-net discretization of $\setSOO$, which would require $\mathcal{O}(\epsilon^{-D})$ evaluations, where $D$ is the ambient dimension. In contrast, EXOTIC achieves a rate governed by the near-optimality dimension $d$, which captures the effective size of near-optimal regions. When $d \ll D$, this leads to significant gains over uniform discretization approaches.
\end{remark}}

\begin{remark}\label{rem: CorollaryStronglyConvex}
Condition~\eqref{eq: ConditionTheoremSC}, intuitively, holds when the optimization solver \(\OPT\) uses a small step size (i.e., when \(\decaySCpgd\) is large). In such a situation, it leads to a slightly weaker bound on the optimality gap. Nevertheless, due to Assumption~\ref{assm: OPT_Convergence_StronglyConvex}, which guarantees exponential convergence of \(\OPT\), the overall convergence rates remain comparable even when this condition is not satisfied. 
Indeed, when \(\iter\) is sufficiently large, both conditions\eqref{eq: Large_n_convex_condition_not_satisfied} and ~\eqref{eq: Large_n_strongly_convex_condition_satisfied} are satisfied. As a result, Corollary~\ref{cor: MainStronglyConvexResult} ensures that the optimality gap decays at a rate of \(\tilde{\mathcal{O}}(n^{-1/d})\), regardless of whether condition~\eqref{eq: ConditionTheoremSC} holds. Interestingly, in terms of \(n,\) the rate matches the convergence rate of the SequOOL algorithm in \cite{bartlett2019simple}, which requires the function \(\outerObj\) to be evaluated exactly (i.e., \(\OPT\) is not required).
\end{remark}

{{}
\subsection{Sufficient conditions for Assumption \ref{assm: Func_Part_Assm}}\label{ssec:SuffCond}
In this section, we provide sufficient conditions under which Assumptions \ref{assm: Func_Part_Assm} holds.
First, we make the following assumption:
\begin{assumption}\label{assm: BigX}
    For every \(\mbf{x}\in X,\) the mapping \(Y\ni\mbf{y}\mapsto f(\mbf{x},\mbf{y})\) is \(L_f-\)Lipschitz. 
\end{assumption}

\begin{proposition}\label{prop: SufficientConditins}
   Suppose that Assumption~\ref{assm: BigX} holds. Furthermore, assume that the hierarchical partitioning \(\{\setSOO_{i,h}\}_{i \in \kTree^h,\, h \in [\hmax]}\) satisfies the following geometric condition: there exist constants \(\alpha > 0\) and \(\beta \in\, ]0,1[\) such that for all \(h \in [\hmax]\) and all \(i \in [\kTree^h]\),
\begin{align}\label{eq: ConditionGeometric}
    \|\mbf{w} - \mbf{w}'\| \leq \alpha \beta^h \quad \text{for all}~ \mbf{w}, \mbf{w}' \in \setSOO_{i,h}.
\end{align}
Then Assumptions~\ref{assm: Func_Part_Assm} is satisfied with \(\nu =\sqrt{d_x+1}L_f\) and \(\rho=\beta\).
\end{proposition}

A proof of Proposition \ref{prop: SufficientConditins} is provided in Appendix \ref{app: SufficientConditions}.

\begin{remark}\label{rem: NCCMM}
    Algorithm~\ref{alg: Algorithm_ComputingSol}, together with its theoretical analysis, can also be applied to non-convex--concave min-max problems. Consider \(
    \min_{\mbf{x}\in \xSet}\max_{\mbf{y}\in \ySet} f(\mbf{x},\mbf{y}),\)
where \(\xSet\) is compact, \(\ySet\) is closed and convex, and \(f\) is non-convex in \(\mbf{x}\) and concave in \(\mbf{y}\). Defining \(
    \tilde f(\mbf{x},\mbf{y}) := - f(\mbf{x},\mbf{y}),\)
the problem can be equivalently written as \(
    \min_{\mbf{x}\in \xSet}\max_{\mbf{y}\in \ySet} f(\mbf{x},\mbf{y})
    =
    - \max_{\mbf{x}\in \xSet}\min_{\mbf{y}\in \ySet} \tilde f(\mbf{x},\mbf{y}).\)
Since \(f\) is concave in \(\mbf{y}\), the function \(\tilde f\) is convex in \(\mbf{y}\). Hence, for every fixed \(\mbf{x}\), the inner problem
is convex.
We then define the induced outer objective \(
    \tilde G(\mbf{x}) := \min_{\mbf{y}\in \ySet} \tilde f(\mbf{x},\mbf{y}),\)
and apply Algorithm~\ref{alg: Algorithm_ComputingSol} to solve \(
    \max_{\mbf{x}\in \xSet} \tilde G(\mbf{x}).\)
The resulting value of the original non-convex--concave min-max problem is obtained by negating the computed max-min value.

The convergence guarantees of Algorithm~\ref{alg: Algorithm_ComputingSol} continue to hold under the corresponding modification of Assumption~\ref{assm: Func_Part_Assm}, where \(G\) and \(W\) are replaced by \(\tilde G\) and \(\xSet\), respectively. In particular, by an argument analogous to Proposition~\ref{prop: SufficientConditins}, the modified assumption holds if the hierarchical partition of \(\xSet\) has geometrically decaying cell diameters and \(\tilde f(\cdot,\mbf{y})\) is uniformly Lipschitz in \(\mbf{x}\) over all \(\mbf{y}\in \ySet\). Under this condition, the value function \(\tilde G\) is Lipschitz on \(\xSet\), and the geometric decay of the partition cells implies the required regularity condition. The proof follows the same steps as the proof of Proposition~\ref{prop: SufficientConditins}.
\end{remark}

}

\section{Numerical Study}\label{sec: Numerics}
In this section, we numerically evaluate the effectiveness of Algorithm~\ref{alg: Algorithm_ComputingSol} for solving convex–non-concave min-max optimization problems. We compare its performance with state-of-the-art methods, including Gradient Descent Ascent (GDA) \cite{jin2020local} and Alternating Gradient Projection (AGP) \cite{xu2023unified}. \footnote{Note that the AGP and GDA algorithms require careful fine-tuning of the hyperparameters in practice. The parameters used in this study are presented in the Appendix \ref{app:NumericalExtras}.}
We begin with a hand-crafted convex–non-concave min-max problem that admits an analytical solution. This example serves as a benchmark to evaluate the performance of GDA, AGP, and our proposed Algorithm~\ref{alg: Algorithm_ComputingSol}. We further validate Algorithm~\ref{alg: Algorithm_ComputingSol} on benchmark problems from the \textsf{SIPAMPL} database \cite{vaz_sipampl}.
Finally, we demonstrate the application of Algorithm~\ref{alg: Algorithm_ComputingSol} in computing security strategies in multi-player games involving three or more players—providing, to the best of our knowledge, the first principled algorithmic approach for these problems. The numerical experiments were implemented in Julia (version 1.11.3) and Python (version 3.10.7).

The partitioning parameter $K$ in Algorithm~\ref{alg: Algorithm_ComputingSol} controls the number of child nodes which are created by splitting a parent node. Choosing a large value of $K$ promotes higher exploration by making the algorithm perform a higher number of function evaluations at a given level in the tree. However, since the total number of nodes in the tree increases exponentially with $K$, we recommended starting with the minimal value $K = 2$ and increasing $K$ only if poor performance is observed at this value. For the purpose of our numerical study, we select $K = 2$. An interesting future research direction could be to design how $K$ can be adaptively selected at every depth $h$ of the tree; however, this is beyond the scope of the current work.




\subsection{Hand-crafted Example}
\label{subsec:hand_crafted_example}
Given positive integers \(\xSetDim\) and \(\ySetDim\), consider the following min-max optimization problem:
\begin{align}\label{eq: toy_example}
    \min_{\mathbf{x} \in [-c, c]^{\xSetDim}} \max_{\mathbf{y} \in [-1, 1]^{\ySetDim}} f(\mathbf{x}, \mathbf{y}) = -(\mathbf{1}^\top \mathbf{y})^3 + (\mathbf{1}^\top \mathbf{x})(\mathbf{1}^\top \mathbf{y}),
\end{align}
where \(c > {3\ySetDim^2}/{\xSetDim}\). Note that the objective function \(f\) is convex (linear) in \(\mathbf{x}\) but non-concave in \(\mathbf{y}\). 

\begin{proposition}\label{prop: HandCraftExample}
   The optimal value of \eqref{eq: toy_example} is equal to \(0.25 \ySetDim^3\). Additionally, the set of optimal solution is given by 
   \[
\left\{(\mbf{x}, \mbf{y}) \in [-c, c]^{\xSetDim}\times [-1, 1]^{\ySetDim} \mid  ~\mathbf{1}^\top \mathbf{x} = 0.75 \ySetDim^2,~  \mathbf{1}^\top \mathbf{y} \in \left\{-\ySetDim, 0.5\ySetDim\right\}\right\}. 
   \]
\end{proposition}
A proof of Proposition \ref{prop: HandCraftExample} is provided in the Appendix \ref{app: proof_handcraft}.

\begin{figure}[h!]
	\centering
	\begin{subfigure}{.49\textwidth}
		\includegraphics[width=\textwidth]{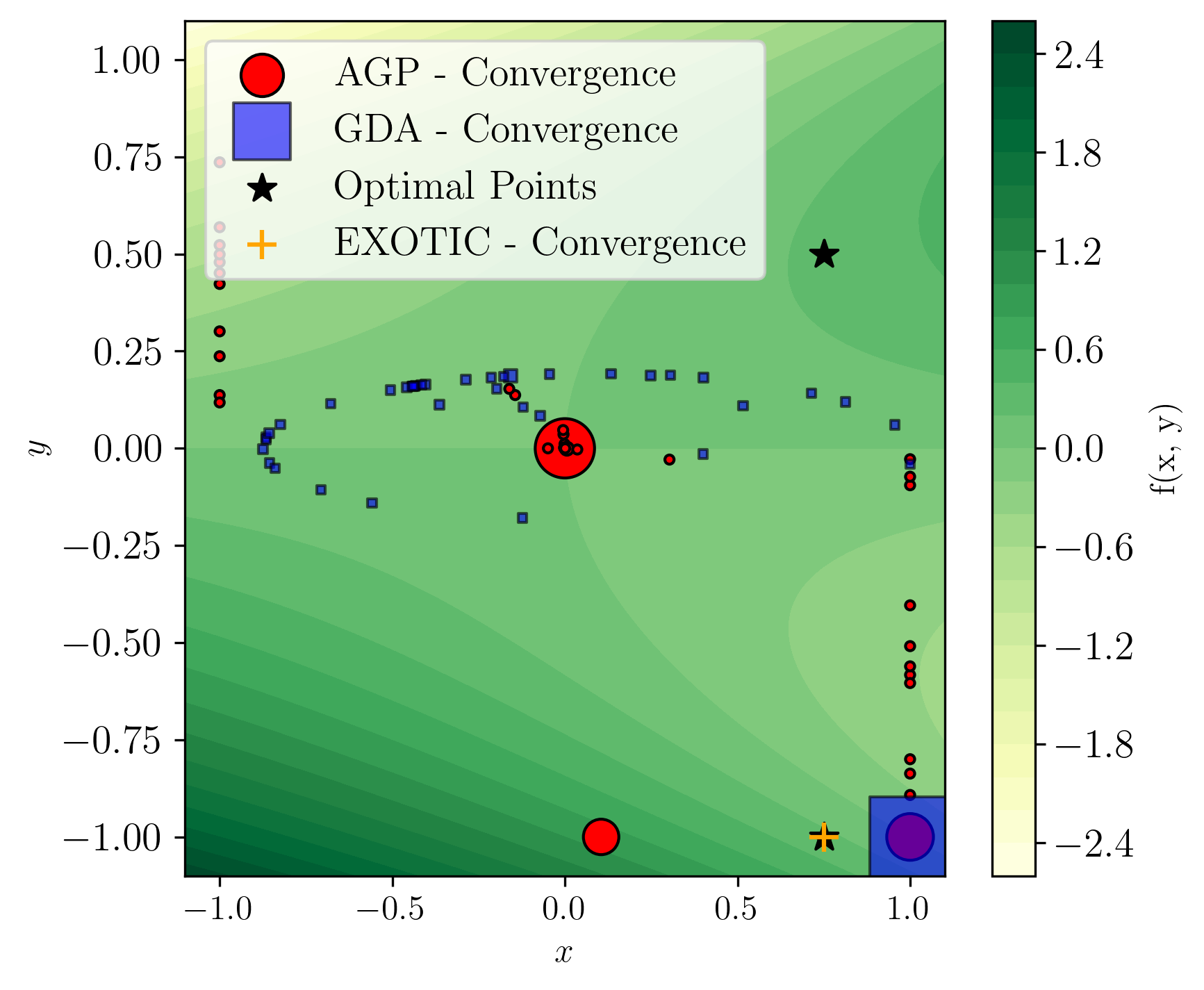}
	\end{subfigure}
	\begin{subfigure}{.49\textwidth}
		\includegraphics[width=\textwidth]{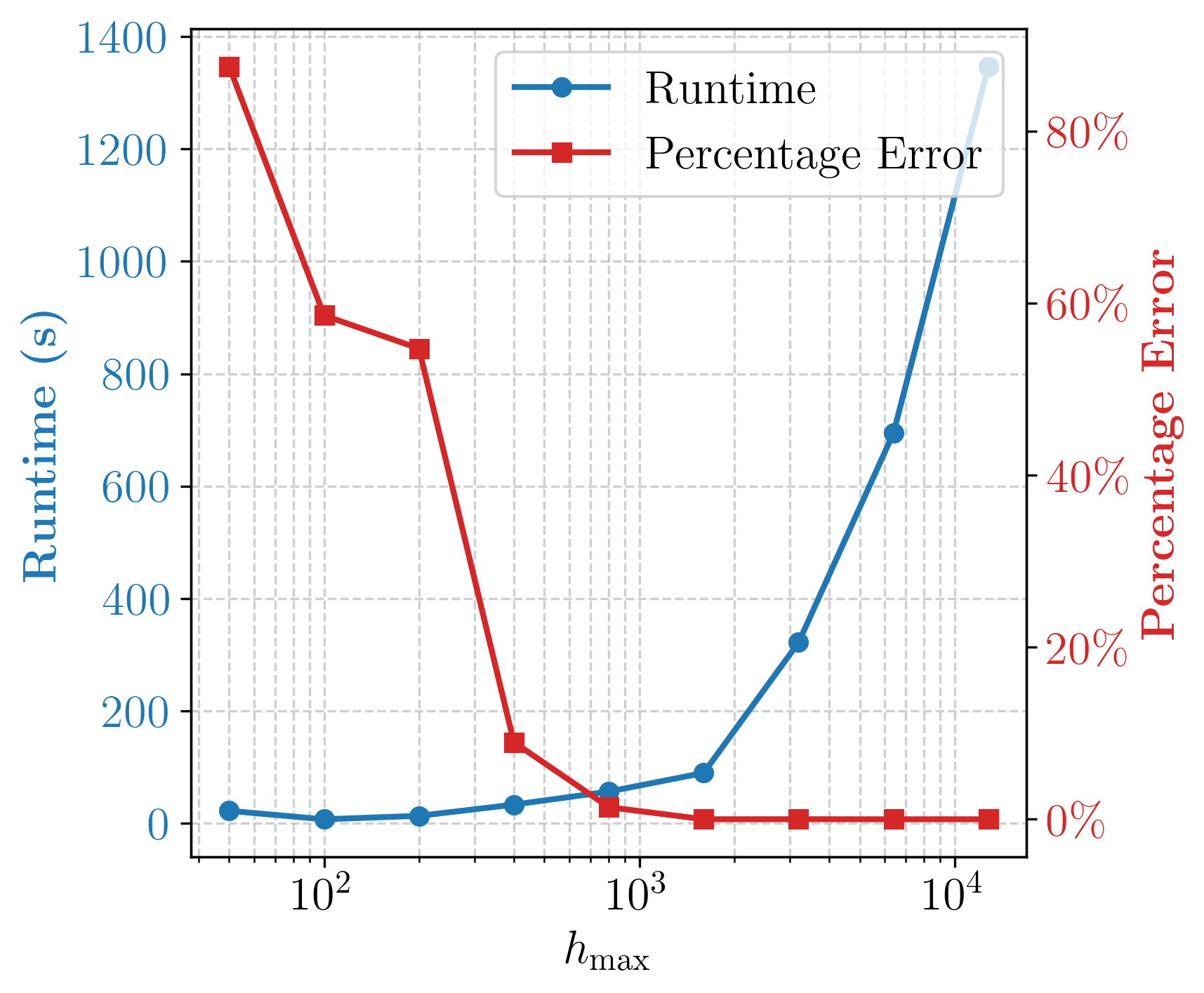}
	\end{subfigure}
	\caption{\small
\textbf{Left:} Solutions obtained by GDA~\cite{jin2020local} (blue squares) and AGP~\cite{xu2023unified} (red circles) for~\eqref{eq: toy_example} with $\xSetDim = \ySetDim = c = 1$, over $125$ runs with random initial conditions. Marker sizes are proportional to the number of runs converging to the corresponding point.  The optimal solution is denoted by black stars. The solution returned by EXOTIC is denoted by yellow plus sign. 
\textbf{Right:} Runtime (seconds) and percentage error, defined as $100\times \frac{|0.25\ySetDim^3 - \hat{G}|}{|0.25\ySetDim^3|}$ where $\hat{G}$ is the value returned by Algorithm~\ref{alg: Algorithm_ComputingSol}, as a function of $h_{\text{max}}$ under EXOTIC for ~\eqref{subsec:hand_crafted_example} with $d_x = d_y = 5$.
}

    \label{fig:AGP_GDA_EXOTIC_benchmarking}
\end{figure}

\par \noindent 

Figure~\ref{fig:AGP_GDA_EXOTIC_benchmarking} (\textbf{Left}) demonstrates that the gradient-based algorithms, GDA and AGP, fail to solve \eqref{subsec:hand_crafted_example} even when $d_x = d_y = c = 1$. In particular, AGP frequently converges to $(0,0)$, which is a stationary point of the dynamics since $\nabla_x f(0,0) = \nabla_y f(0,0) = 0$. On the other hand, GDA typically converges to $(1,1)$, which is a stationary point of GDA dynamics. The smaller blue squares near the center in Figure~\ref{fig:AGP_GDA_EXOTIC_benchmarking} (\textbf{Left}) correspond to cycles in the GDA trajectories. Meanwhile, EXOTIC is able to converge to an optimal solution. An animated video presenting the nodes explored by EXOTIC is available at:   \url{https://tinyurl.com/yfuzurs9}. 
Figure~\ref{fig:AGP_GDA_EXOTIC_benchmarking} (\textbf{Right}) illustrates that the runtime of EXOTIC increases proportionally with $h_{\text{max}}$, as a larger tree is required. Moreover, as $h_{\text{max}}$ increases, the relative percentage error decreases rapidly and eventually approaches zero.


Table~\ref{tab:dimension_results} presents the performance of EXOTIC under various parameter settings for~\eqref{subsec:hand_crafted_example}. Since the runtime of the algorithm scales linearly with $h_{\text{max}}$, a larger value of $h_{\text{max}}$ should be selected only when the effective problem dimension $d_x d_y$ increases. Empirically, we find that setting $h_{\text{max}} = 10^2 d_x d_y$ suffices to achieve a relative percentage error below $0.001\%$. The table further demonstrates that EXOTIC attains high-accuracy solutions even for problems with $d_y = 100$.  
As $d_x$ grows, however, each invocation of the convex optimization solver \(\OPT\) becomes computationally expensive. In such cases, it is advisable to employ specialized convex optimization solvers optimized for high-dimensional settings to maintain low run-times.

\begin{table}[h!]

\centering

\begin{subtable}{\linewidth}
\centering
\caption{Low-dimensional}
\begin{tabular}{ccccccc}
\toprule
$d_x$ & $d_y$ & $\mathrm{Opt}_{\text{true}}$ & $\hat{G}$ & $\% ~ \text{Error}$ & $h_{\max}$ & Runtime (s) \\
\midrule
1 & 1 & 0.25 & 0.25 & 0 & $10^2$ & $5 \times 10^0$ \\
1 & 2 & 2 & 2 & 0 & $2 \times 10^2$ & $6 \times 10^0$  \\
2 & 1 & 0.25 & 0.25 & 0 & $2 \times 10^2$ & $7 \times 10^0$ \\
3 & 2 & 2 & 1.998 & $0.1\%$ & $4 \times 10^2$ & $3 \times 10^1$ \\
2 & 3 & 6.75 & 6.75 & 0 & $5 \times 10^2$ & $2 \times 10^1$ \\
3 & 3 & 6.75 & 6.738 & $0.18\%$ & $6 \times 10^2$ & $4.8 \times 10^1$ \\
\bottomrule
\end{tabular}
\end{subtable}

\vspace{0.5em} 

\begin{subtable}{\linewidth}
\centering
\caption{Moderate-dimensional}
\begin{tabular}{ccccccc}
\toprule
$d_x$ & $d_y$ & $\mathrm{Opt}_{\text{true}}$ & $\hat{G}$ & $\% ~ \text{Error}$ & $h_{\max}$ & Runtime (s) \\
\midrule
5 & 5 & 31.25 & 31.24 & $0.03\%$ & $1.6 \times 10^3$ & $9 \times 10^1$ \\
3 & 10 & 250 & 249.94 & $0.02\%$ & $2 \times 10^3$ & $1.9 \times 10^2$ \\
10 & 3 & 6.75 & 6.74 & $0.15\%$ & $2 \times 10^3$ & $1 \times 10^2$ \\
3 & 20 & 2000 & 1999.75 & $0.01\%$ & $4 \times 10^3$ & $3 \times 10^2$ \\
20 & 3 & 6.75 & 6.736 & $0.2\%$ & $4 \times 10^3$ & $4 \times 10^2$ \\
\bottomrule
\end{tabular}
\end{subtable}

\vspace{0.5em}

\begin{subtable}{\linewidth}
\centering
\caption{High-dimensional}
\begin{tabular}{ccccccc}
\toprule
$d_x$ & $d_y$ & $\mathrm{Opt}_{\text{true}}$ & $\hat{G}$ & $\% ~ \text{Error}$ & $h_{\max}$ & Runtime (s) \\
\midrule
2 & 100 & $2.5 \times 10^5$ & $2.48 \times 10^5$ & $0.8\%$ & $1.3 \times 10^4$ & $1.5 \times 10^3$  \\
20 & 20 & 2000 & 1949.5 & $2.5\%$ & $1.3 \times 10^4$ & $2.4 \times 10^3$ \\
\bottomrule
\end{tabular}
\end{subtable}

\caption{\small 
Performance of EXOTIC on \eqref{eq: toy_example} for various problem parameters. 
$\mathrm{Opt}_{\text{true}}$ denotes the ground-truth optimal value from Proposition~\ref{prop: HandCraftExample}, and $\hat{G}$ is the value computed by EXOTIC. 
The relative percentage error is defined as $\%~\mathrm{Error} = 100 \times \frac{|\mathrm{Opt}_{\text{true}} - \hat{G}|}{\mathrm{Opt}_{\text{true}}}$. 
Here, $h_{\max}$ is the maximum depth of the tree expansion, and runtime is reported in seconds.
}

\label{tab:dimension_results}

\end{table}

{
\subsection{Benchmark Problems from the \texttt{SIPAMPL} Database}\label{ssec:BenchmarkSIPAMPL} Next, we evaluate the performance of our proposed algorithm on benchmark problems from the \texttt{SIPAMPL} database \cite{vaz_sipampl}. Specifically, we consider the following three problems: \texttt{leon10} \cite{LEON200078}, \texttt{hettich4} \cite{hettich_1979}, and \texttt{hettich5} \cite{hettich_1979}:
\begin{align}
    \min_{(x_1, x_2) \in \mathbb{R}^2} \max_{y \in [0, 2]} 
    \left|x_1y + x_2\exp(y) - y^2\right|, 
    \tag{\texttt{leon10} \cite{LEON200078}}
\end{align}

\begin{align}
    \min_{x \in \mathbb{R}} \max_{y \in [0, 0.75]} 
    \left|1 - 8xy + 8x^2\right|, 
    \tag{\texttt{hettich4} \cite{hettich_1979}}
\end{align}

\begin{align}
    \min_{(x_1, x_2) \in \mathbb{R}^2} 
    \max_{(y_1, y_2) \in (0, \sqrt{2}]^2} 
    \left|
    y_1^2 + y_2^2 
    - \left(
    x_1\sqrt{y_1^2 + y_2^2}
    + x_2\exp\left(\sqrt{y_1^2 + y_2^2}\right)
    \right)
    \right|. 
    \tag{\texttt{hettich5} \cite{hettich_1979}}
\end{align}

Table~\ref{table:CSIP_benchmark} compares the performance of EXOTIC with GDA and AGP on these benchmark problems.
EXOTIC consistently attains errors on the order of $10^{-4}$ in the worst case, demonstrating reliable accuracy across problem instances. In contrast, AGP and GDA fail to converge in most cases, and even when convergence occurs, the resulting solutions exhibit poor quality. 
}

\begin{table}[h!]
\centering
\begin{tabular}{lllll}
\toprule
\textbf{Problem} & \multicolumn{4}{c}{\textbf{Values}}\\
    \cmidrule(r{4pt}){2-5}
    & \textbf{Best reported} & \textbf{EXOTIC} & \textbf{GDA} & \textbf{AGP} \\ \midrule
      leon10  (\cite{LEON200078}) &  $0.53825$    &   $0.5382$  & \DNconv  & \DNconv  \\
      hettich4 (\cite{hettich_1979})  & $1.0$ & $0.9999$   & \DNconv & \DNconv \\
      hettich5 (\cite{hettich_1979})  & $0.538$ & $0.53825$   & 2.00 & \DNconv \\
      \bottomrule
\end{tabular}
\caption{\small Performance of EXOTIC compared with GDA and AGP on the benchmark {convex--non-concave min-max} problems from \textsf{SIPAMPL} database. A \DNconv ~ sign means that either the algorithm did not converge or that the convergent point was very far from the optimal point.
}
\label{table:CSIP_benchmark}
\end{table}

\subsection{Application to Multi-player Games}\label{ssec:MultiplayerGames}
Consider a multi-player game comprising of \(N\) players. The (finite)set of actions of player \(i\in [N]\) is denoted \(A_i\). The joint action of all players is denoted by \(A = \times_{i\in [N]} A_i\). 
The cost experienced by player \(i\in [N]\) depends on the joint action of all players, and is denoted by \(c_i: A \rightarrow \mathbb{R}\).   
Every player \(i\in[N]\), selects a strategy \(x_i \in \Delta(A_i)\), which is a probability distribution over the set \(A_i.\)
The \emph{security value} of player \(i\) is computed as follows
\begin{align*}
\min_{x_i \in \Delta(A_i)}\;
\max_{x_{-i} \in \prod_{j\ne i}\Delta(A_j)} c_i(x_i, x_{-i}) = 
\sum_{a=(a_i)_{i\in [N]}\in A}\;\,\Bigg(\prod_{i\in [N]} x_k(a_k)\Bigg)\, c_i(a) .
\end{align*}
The security value, intuitively, denotes the best possible value that player \(i\) gets when other players are adversarial. 
Observe that the problem of computing security value is convex in \(x_i\), and non-concave (multi-linear) in \(x_{-i}\). 

To evaluate the performance of Algorithm \ref{alg: Algorithm_ComputingSol}, we create an instance of 3-player game. For every \(i\in [3],\) we take $A_i = \{\alpha,\beta\}$. Here, we focus on computing the security value of player \(1\), where its cost function, $c_1: A \rightarrow \mathbb{R}$, is provided in Table \ref{table:cost_structure}. 
\begin{table}[h]
\centering
\begin{subtable}[t]{0.45\textwidth}
\centering
\begin{tabular}{|l|l|l|c|}
\hline
$a_1$ & $a_2$ & $a_3$ & $c_1(a_1,a_2,a_3)$ \\ \hline
$\actZero$     & $\actZero$   & $\actZero$     & 2.1               \\ \hline
$\actZero$     & $\actZero$     & $\actOne$     & 1.2               \\ \hline
$\actZero$     & $\actOne$     & $\actZero$     & 1.5               \\ \hline
$\actZero$     & $\actOne$     & $\actOne$     & 1.6               \\ \hline
\end{tabular}
\end{subtable}%
\hspace{0.01\textwidth} 
\begin{subtable}[t]{0.45\textwidth}
\centering
\begin{tabular}{|l|l|l|c|}
\hline
$a_1$ & $a_2$ & $a_3$ & $c_1(a_1,a_2,a_3)$ \\ \hline
$\actOne$     & $\actZero$      & $\actZero$      & 1.5               \\ \hline
$\actOne$     & $\actZero$     & $\actOne$     & 0.4               \\ \hline
$\actOne$     & $\actOne$     & $\actZero$     & 1.5               \\ \hline
$\actOne$     & $\actOne$     & $\actOne$     & 1.7               \\ \hline
\end{tabular}
\end{subtable}
\caption{\small Cost Structure for the Example Problem}
\label{table:cost_structure}
\end{table}
For this problem, the solution returned by EXOTIC is  $x_1^{\algm} = [0.28, 0.72]$, $x_2^{\algm} = [0, 1]$, $x_3^{\algm} = [0 , 1]$, and $c_1(x_1^{\algm}, x_2^{\algm}, x_3^{\algm})= 1.672$ where $x_2^{\algm}$ and $x_3^{\algm}$ are the worst-case adversarial strategies of other players. Figure \ref{fig:Security_EXOTIC_vs_AGP_vs_GDA} compares the performance of EXOTIC with that of GDA and AGP. We observe that the EXOTIC is more secure than GDA and AGP in presence of strong adversaries. Moreover, the adversarial actions generated from AGP and GDA are benign and yield less cost if player uses the strategy from EXOTIC. 

\begin{figure}[h!]
	\centering
	\begin{subfigure}{.7\textwidth}
		\includegraphics[width=\textwidth]{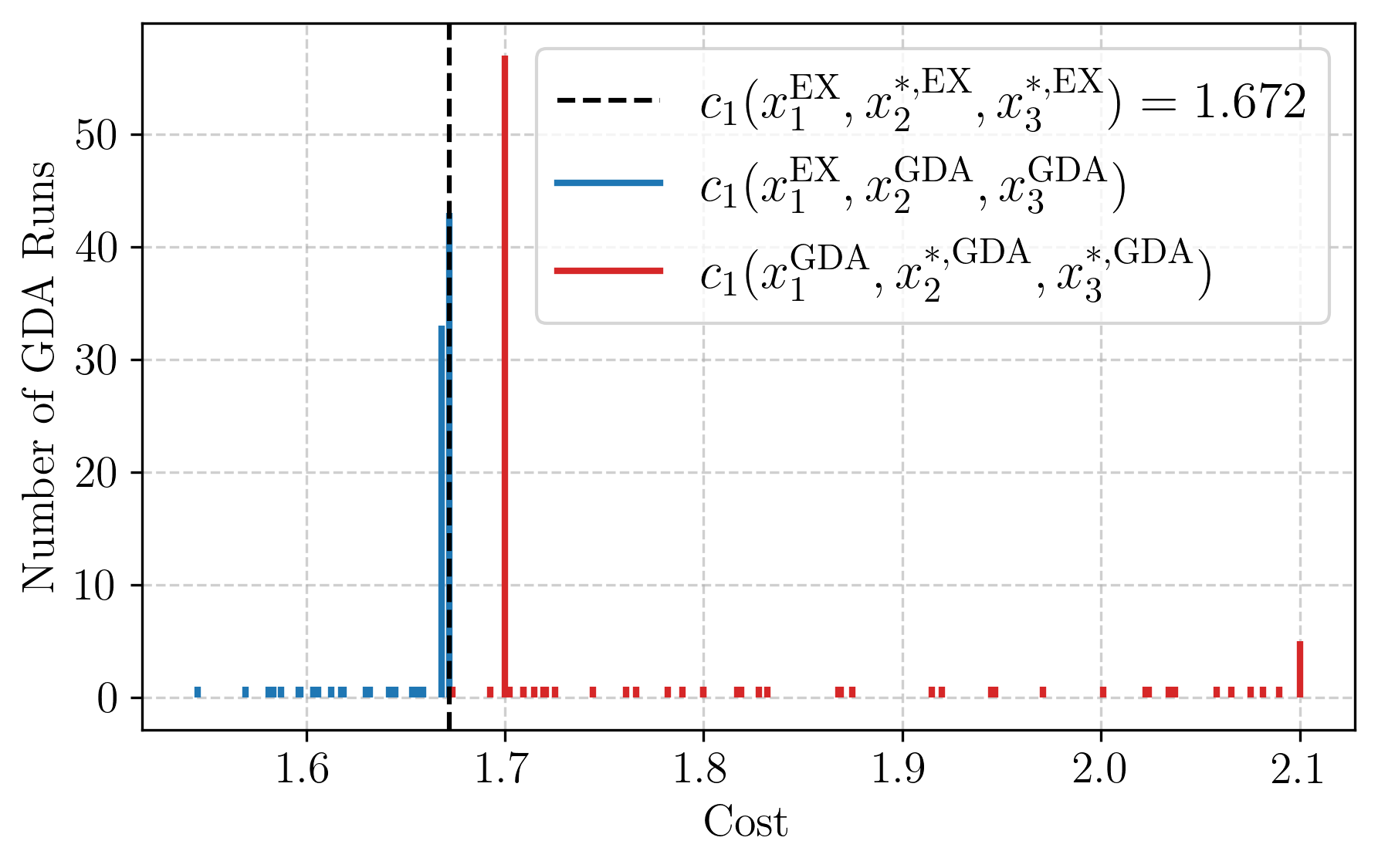}
	\end{subfigure}
	\begin{subfigure}{.7\textwidth}
		\includegraphics[width=\textwidth]{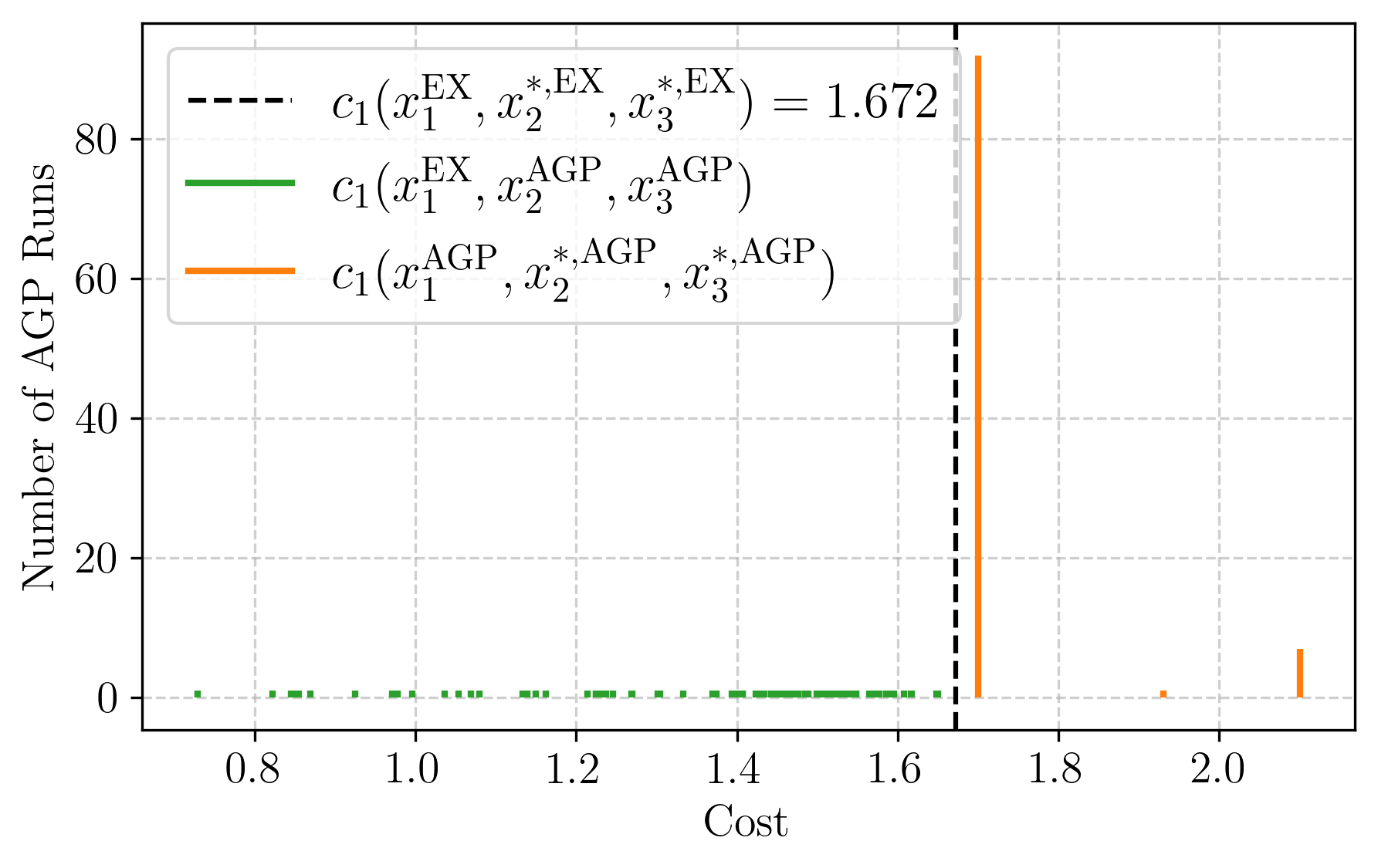}
	\end{subfigure}
	\caption{\small Comparison of converged strategies of Algorithm \ref{alg: Algorithm_ComputingSol} with $100$ different runs of GDA and AGP for computing the security value in 3-player game. The black dotted line in each plot corresponds to the security value computed by EXOTIC, which is $c_1(x_1^{\algm}, x_2^{\algm}, x_3^{\algm})= 1.672$. For every run, we compute the terms $c_1(x_1^{\algm}, x_2^{\text{Alg}}, x_3^{\text{Alg}})$ and $c_1(x_1^{\text{Alg}}, x_2^{\ast, \text{Alg}}, x_3^{\ast, \text{Alg}})$, where \((x_2^{\ast, \text{Alg}}, x_3^{\ast, \text{Alg}})\in \arg\max_{x_2, x_3} c_1(x_1^{\text{Alg}}, x_2, x_3)\) for each $\text{Alg} \in \{\text{GDA}, \text{AGP}\}$. The former term corresponds to the cost incurred by player $1$ when they play according to EXOTIC while their opponents play adversarial strategies according to $\text{Alg}$. The latter corresponds to the cost when player $1$ plays according to $\text{Alg}$, while their opponents play the worst-case adversarial strategies. The worst-case strategies for a given action of player $1$ were computed through brute force enumeration. For every run, we observe $c_1(x_1^{\algm}, x_2^{\text{Alg}}, x_3^{\text{Alg}}) \leq c_1(x_1^{\algm}, x_2^{\ast, \algm}, x_3^{\ast, \algm})$, which shows that the adversarial strategies computed by $\text{Alg}$ are benign compared to $\algm$.  Further, $c_1(x_1^{\algm}, x_2^{\ast, \algm}, x_3^{\ast, \algm}) \leq c_1(x_1^{\text{Alg}}, x_2^{\ast, \text{Alg}}, x_3^{\ast, \text{Alg}})$, which shows that EXOTIC is more secure than both GDA and AGP in presence of strong adversaries. 
} 
    \label{fig:Security_EXOTIC_vs_AGP_vs_GDA}
\end{figure}





\appendix

\section{Proof of Proposition \ref{prop: ProblemReformulate}}\label{app: ProofMinmaxToMaxmin}
{Before stating the proof, we define a mapping \(\mathbb{R}\times X \ni (t,\mbf{x}) = \tilde{\mbf{x}} \mapsto g(\tilde{\mbf{x}}) = t\in \mathbb{R}\) and a set-valued map \(\ySet^{\xSetDim+1}=\setSOO \ni \mbf{w} = [\mbf{y}_1^\top, \mbf{y}_2^\top, \dots, \mbf{y}_{\xSetDim+1}^\top]^\top \mapsto \setmap(\mbf{w})\) as 
\begin{equation}
    \label{eq:graph}
\setmap(\mbf{w}) = \{\tilde{\mbf{x}} = (t, \mbf{x}) \in \tRange \times \xSet ~ | ~ f(\mbf{x},\mbf{y}_i) \le t \quad \text{for} \quad i \in 1, 2, \dots, \xSetDim+1 \}.
\end{equation} }
We reformulate \eqref{eq:min_max} into the following semi-infinite problem (SIP):
        \begin{align}
\label{eq:min_max_to_CSIP}
&\min_{(t, \mbf{x}) \in \tRange \times \xSet}  \quad \left\{ t  \mid ~ f(\mbf{x},\mbf{y}) \le t \quad \text{for all} \quad \mbf{y} \in \ySet \right\}. 
\end{align}
Notice that \eqref{eq:min_max_to_CSIP} is convex SIP 
because the objective function is linear and therefore convex, and for every \(\mbf{y} \in \ySet\), the sub-level set \(\{(t, \mbf{x}) |  f(\mbf{x}, \mbf{y}) \leq t\} \subseteq \mathbb{R} \times \xSet\) is convex.

 
Also, note that \eqref{eq:min_max_to_CSIP} satisfies all the conditions required \cite[Theorem 1]{das2022near}. This ensures that the optimal value of \eqref{eq:min_max_to_CSIP} is same as that of the following problem:

{\small\begin{align}
\label{eq:CSIP_to_supinf}
\sup_{\mbf{w}\in \setSOO}\inf_{\tilde{\mbf{x}}\in \setmap(\mbf{w})} \innerObj(\tilde{\mbf{x}}) = \sup_{\mbf{w} = (\mbf{y}_i)_{i\in [\xSetDim+1]}\in \setSOO}& \left\{ \inf_{ (t, \mbf{x}) \in \tRange \times \xSet}  \Bigl\{ t \,\Big|\, f(\mbf{x},\mbf{y}_i) \le t ~ \forall i\in [\xSetDim+1]\Bigr\} \right\}.
\end{align}}

Finally, we claim that the ``sup-inf'' formulation in~\eqref{eq:CSIP_to_supinf} can be converted into a ``max-min'' formulation. To justify this conversion, we establish: 
\begin{enumerate}[labelwidth=*,align=left, widest=iii, leftmargin=*]
    \item[(F1)] For any $\mbf{w} \in \setSOO$, \(
    \inf_{\tilde{\mbf{x}} \in \setmap(\mbf{w})} \innerObj(\tilde{\mbf{x}}) = \min_{\tilde{\mbf{x}} \in \setmap(\mbf{w})} \innerObj(\tilde{\mbf{x}}),
    \) and 
    \item[(F2)] $\sup_{\mbf{w} \in \setSOO} \outerObj(\mbf{w}) = \max_{\mbf{w} \in \setSOO} \outerObj(\mbf{w})$.
\end{enumerate}

\noindent \textbf{(F1)} follows directly from \cite[Theorem 1.9]{rockafellar1998variational}, which states that if a function is lower semi-continuous, level-bounded, and proper, then the infimum is attained. Since the function \(\R \times \xSet \ni \tilde{\mbf{x}} \mapsto \innerObj(\tilde{\mbf{x}})\) is continuous, it is also lower semi-continuous and proper.
For any \(\alpha \in \mathbb{R}\), \(
    \left\{ \tilde{\mathbf{x}} \in \mathbb{R} \times \xSet \;\middle|\; \innerObj(\tilde{\mathbf{x}}) \leq \alpha \right\}
    \subseteq \bigcup_{\mathbf{y} \in \ySet} \left\{ \mathbf{x} \in \xSet \;\middle|\; \obj(\mathbf{x}, \mathbf{y}) \leq \alpha \right\},
\)
which is bounded since we assume, in \eqref{eq:min_max}, that \(\ySet\) is compact and, for every \(\mathbf{y} \in \ySet\), the function \(\R^{\xSetDim}\ni \mathbf{x} \mapsto \obj(\mathbf{x}, \mathbf{y})\) is coercive.

\noindent \textbf{(F2)} is also established using \cite[Theorem 1.9]{rockafellar1998variational} by verifying that \(\outerObj(\cdot)\) is lower semi-continuous, proper, and level-bounded. Recall that for any \(\mbf{w} \in \setSOO\), we have \(
\outerObj(\mbf{w}) = \inf_{\tilde{\mbf{x}} \in \setmap(\mbf{w})} \innerObj(\tilde{\mbf{x}}).
\)
We begin by noting that \(\outerObj(\cdot)\) is level-bounded. This follows immediately from the compactness of \(\setSOO\).
Next, we establish that \(\outerObj(\cdot)\) is lower semi-continuous and proper. Fix an arbitrary \(\mbf{w} = [\mbf{y}_1^\top, \mbf{y}_2^\top, \ldots, \mbf{y}_{\xSetDim+1}^\top]^\top \in \ySet^{\xSetDim+1} = \setSOO.\) Observe that \(
\outerObj(\mbf{w}) = \min_{\mbf{x} \in \xSet} F(\mbf{x}, \mbf{w})\),
where 
\(F(\mbf{x}, \mbf{w}) := \max_{i \in [\xSetDim+1]} f(\mbf{x}, \mbf{y}_i).
\)
Since \(\outerObj(\cdot)\) is the optimal value of a parametric nonlinear optimization problem, it suffices by \cite[Theorem 1.17]{rockafellar1998variational} to verify the following two conditions:
\begin{enumerate}[labelwidth=*,align=left, widest=iii, leftmargin=*]
    \item[(i)] The function \(\mbf{x} \mapsto F(\mbf{x}, \mbf{w})\) is proper and lower semi-continuous for every \(\mbf{w} \in \setSOO\); and 
    \item[(ii)] \(F(\mbf{x}, \mbf{w})\) is level-bounded in \(\mbf{x}\), locally uniformly in \(\mbf{w}\).
\end{enumerate}

\noindent (i) holds because \(\mbf{x} \mapsto f(\mbf{x}, \mbf{y})\) is continuous for every \(\mbf{y} \in \ySet\), and the pointwise maximum of finitely many continuous functions is continuous.

\noindent To establish (ii), it suffices to show that the set \(\bigcup_{\mathbf{w} \in \setSOO} \left\{ \mathbf{x} \in \mathcal{X} \mid F(\mathbf{x}, \mathbf{w}) \leq \alpha \right\}\) is bounded for every \(\alpha \in \mathbb{R}\). Indeed, for any \(\alpha \in \mathbb{R}\), the set \(
\cup_{\mathbf{w} \in \setSOO} \left\{ \mathbf{x} \in \xSet \mid F(\mathbf{x}, \mathbf{w}) \leq \alpha \right\}\) is contained in the set \( 
\cup_{\mathbf{y} \in \ySet} \left\{ \mathbf{x} \in \xSet \mid \obj(\mathbf{x}, \mathbf{y}) \leq \alpha \right\},
\)
which is bounded since, as assumed in \eqref{eq:min_max}, the set \(\ySet\) is compact and, for every \(\mathbf{y} \in \ySet\), the function \(\mathbf{x} \mapsto \obj(\mathbf{x}, \mathbf{y})\) is coercive.

{\section{Proof of Proposition \ref{prop:solution_and_stationary_correspondence}}\label{app: StationaryPoint}
Since \(F(\mbf{x},\mbf{w})\le \Phi(\mbf{x})\) for every \(\mbf{x}\in X\) and every \(\mbf{w}\in Y^{d_x+1}\), we have \(
    \min_{\mbf{x}\in X}F(\mbf{x},\mbf{w}^\ast)
    \le
    \min_{\mbf{x}\in X}\Phi(\mbf{x}).\)
By Proposition 2.1 and the optimality of \(\mbf{w}^\ast\),
\[
    \min_{\mbf{x}\in X}F(\mbf{x},\mbf{w}^\ast)
    =
    \max_{\mbf{w}\in Y^{d_x+1}}
    \min_{\mbf{x}\in X}F(\mbf{x},\mbf{w})
    =
    \min_{\mbf{x}\in X}\Phi(\mbf{x}).
\]

We first prove the result under condition (i). Let \(\bar{\mbf{x}}\) be the unique minimizer of \(\Phi\). Then
\[
    F(\bar{\mbf{x}},\mbf{w}^\ast)
    \le
    \Phi(\bar{\mbf{x}})
    =
    \min_{\mbf{x}\in X}\Phi(\mbf{x})
    =
    \min_{\mbf{x}\in X}F(\mbf{x},\mbf{w}^\ast).
\]
Since \(F(\bar{\mbf{x}},\mbf{w}^\ast)\) cannot be smaller than
\(\min_{\mbf{x}\in X}F(\mbf{x},\mbf{w}^\ast)\), equality holds. Hence \(\bar{\mbf{x}}\) is also a minimizer of \(F(\cdot,\mbf{w}^\ast)\). By uniqueness, \(\bar{\mbf{x}}=\mbf{x}^\ast\), and therefore \(\mbf{x}^\ast\) minimizes \(\Phi\).

Next suppose condition (ii) holds. Since
\(\mbf{x}^\ast\in\arg\min_{\mbf{x}\in X}F(\mbf{x},\mbf{w}^\ast)\), it holds that \(
    F(\mbf{x}^\ast,\mbf{w}^\ast)
    =
    \min_{\mbf{x}\in X}F(\mbf{x},\mbf{w}^\ast).\)
Using exactness at \(\mbf{x}^\ast\) and value equivalence,
\[
    \min_{\mbf{x}\in X}\Phi(\mbf{x})
    \le
    \Phi(\mbf{x}^\ast)
    =
    F(\mbf{x}^\ast,\mbf{w}^\ast)
    =
    \min_{\mbf{x}\in X}F(\mbf{x},\mbf{w}^\ast)
    =
    \min_{\mbf{x}\in X}\Phi(\mbf{x}).
\]
Thus all inequalities are equalities, and \(\Phi(\mbf{x}^\ast)=\min_{\mbf{x}\in X}\Phi(\mbf{x})\). Hence \(\mbf{x}^\ast\) is a global minimizer of the original minimax problem.

We now prove the stationary-point correspondence. First-order stationarity of \(\mbf{x}^\ast\) for \(\min_{\mbf{x}\in X}F(\mbf{x},\mbf{w}^\ast)\) means \(
    0\in
    \partial_{\mbf{x}}F(\mbf{x}^\ast,\mbf{w}^\ast)
    +N_X(\mbf{x}^\ast),\)
where \(N_X(\mbf{x}^\ast)\) denotes the normal cone of \(X\) at \(\mbf{x}^\ast\). By first-order exactness, \(    \partial_{\mbf{x}}F(\mbf{x}^\ast,\mbf{w}^\ast)
    =
    \partial \Phi(\mbf{x}^\ast).\)
Therefore, \(
    0\in
    \partial \Phi(\mbf{x}^\ast)+N_X(\mbf{x}^\ast),\)
which is the first-order stationarity condition for \(\min_{\mbf{x}\in X}\Phi(\mbf{x})\).

The converse follows identically. This completes the proof.

\section{Proof of Proposition  \ref{prop:SionMinimax}}\label{app: SionMinimax}
{
\begin{proof}
Define \(
\alpha := \min_{\mbf{x}\in \xSet}\max_{\mbf{y}\in \ySet} f(\mbf{x},\mbf{y}).\) By the minimax inequality, we now that \(\max_{\mbf{y}\in \ySet}\min_{\mbf{x}\in \xSet} f(\mbf{x},\mbf{y}) \le \alpha.\) Thus, the only non-trivial fact to establish is that \(\alpha\leq\max_{\mbf{y}\in \ySet}\min_{\mbf{x}\in \xSet} f(\mbf{x},\mbf{y}).\)

\noindent From Proposition 2.1, we know that \(
\alpha
=
\max_{\bar{\mbf{y}} \in \ySet^{\xSetDim+1}} 
\min_{\mbf{x}\in \xSet} \max_{i\in[\xSetDim+1]} f(\mbf{x},\mbf{y}_i). 
\)
Therefore, for any $\varepsilon > 0$, there exists \(
(\mbf{y}_1,\dots,\mbf{y}_{\xSetDim+1}) \in \ySet^{\xSetDim+1}\)
such that 
\[
\min_{\mbf{x}\in \xSet} \max_{i\in[\xSetDim+1]} f(\mbf{x},\mbf{y}_i)
\ge \alpha - \varepsilon.\]

Consider the following epigraph reformulation of \(
\min_{\mbf{x}\in \xSet} \max_{i\in[\xSetDim+1]} f(\mbf{x},\mbf{y}_i):\)
\begin{equation}\label{eq:reform}
\begin{aligned}
\min_{\mbf{x}\in \xSet,\; t\in \mathbb{R}} \quad  t \quad \quad 
\text{s.t.}\quad  f(\mbf{x},\mbf{y}_i)\le t,\qquad \forall \ i\in[\xSetDim+1].
\end{aligned}
\end{equation}
This is a convex optimization problem with a nonempty feasible set. Moreover, Slater’s condition holds: for any $\tilde{\mbf{x}}\in \xSet$, choosing $\tilde t > \max_i f(\tilde{\mbf{x}},\mbf{y}_i)$ yields a strictly feasible point.
The Lagrangian of \eqref{eq:reform} is
\[
L(\mbf{x},t,\lambda)
=
t+\sum_{i=1}^{\xSetDim+1}\lambda_i\bigl(f(\mbf{x},\mbf{y}_i)-t\bigr)
=
\sum_{i=1}^{\xSetDim+1}\lambda_i f(\mbf{x},\mbf{y}_i)
+
\Bigl(1-\sum_{i=1}^{\xSetDim+1}\lambda_i\Bigr)t,
\]
where $\lambda_i\ge 0$. Hence, the dual function is finite only when \(
\lambda \in \Delta_{\xSetDim+1}:=
\Bigl\{\lambda\in\mathbb{R}^{\xSetDim+1}_+:\sum_{i=1}^{\xSetDim+1}\lambda_i=1\Bigr\}.\)
Therefore, the dual problem is \(
\max_{\lambda\in\Delta_{\xSetDim+1}}
\min_{\mbf{x}\in \xSet}
\sum_{i=1}^{\xSetDim+1}\lambda_i f(\mbf{x},\mbf{y}_i).\)

By strong duality, there exists $\lambda^\star\in \Delta_{\xSetDim+1}$ such that \(
\min_{\mbf{x}\in \xSet} \max_{i\in[\xSetDim+1]} f(\mbf{x},\mbf{y}_i)\) \(
=
\min_{\mbf{x}\in \xSet}
\sum_{i=1}^{\xSetDim+1}\lambda_i^\star f(\mbf{x},\mbf{y}_i).\)
Hence, \(
\min_{\mbf{x}\in \xSet}
\sum_{i=1}^{\xSetDim+1}\lambda_i^\star f(\mbf{x},\mbf{y}_i)
\ge \alpha - \varepsilon.\)
Define \(
\hat{\mbf{y}} := \sum_{i=1}^{\xSetDim+1}\lambda_i^\star \mbf{y}_i.\) Note that \(\hat{\mbf{y}}\in Y\) as \(Y\) is a convex set.
By concavity of $f(\mbf{x},\cdot)$,
\[
f(\mbf{x},\hat{\mbf{y}})
=
f\Bigl(\mbf{x},\sum_{i=1}^{\xSetDim+1}\lambda_i^\star \mbf{y}_i\Bigr)
\ge
\sum_{i=1}^{\xSetDim+1}\lambda_i^\star f(\mbf{x},\mbf{y}_i),
\quad \forall \mbf{x} \in \xSet.
\]
Taking the minimum over $\mbf{x}\in \xSet$ gives
\[
\min_{\mbf{x}\in \xSet} f(\mbf{x},\hat{\mbf{y}})
\ge
\min_{\mbf{x}\in \xSet}
\sum_{i=1}^{\xSetDim+1}\lambda_i^\star f(\mbf{x},\mbf{y}_i)
\ge \alpha - \varepsilon.
\]
Thus, \(
\max_{\mbf{y}\in \ySet}\min_{\mbf{x}\in \xSet} f(\mbf{x},\mbf{y})
\ge \alpha - \varepsilon.\)
Since $\varepsilon>0$ is arbitrary, we conclude that \(
\max_{\mbf{y}\in \ySet}\min_{\mbf{x}\in \xSet} f(\mbf{x},\mbf{y})
\ge \alpha.\) This concludes the proof.
\end{proof}}
}

\section{Proof of Proposition \ref{prop: Num_Iterations}}\label{app: NumIterations}
    The number of iterations of the optimization solver \(\OPT\) in the initialization phase is equal to \(\kTree\hmax\).
    In the tree-expansion phase (\textsf{lines 2-11} in Algorithm 1), the number of calls to \(\OPT\) are bounded as follows 
    \begin{align}
        \label{eq: TreeExpansionCompute}&\sum_{h=1}^{\hmax}\sum_{m=1}^{\lfloor \hmax/h \rfloor}\kTree\left\lfloor \frac{\hmax}{hm} \right\rfloor  \leq \sum_{h=1}^{\hmax}\sum_{m=1}^{\lfloor \hmax/h \rfloor}\kTree \frac{\hmax}{hm} = \sum_{h=1}^{\hmax}\frac{\kTree\hmax}{h}\sum_{m=1}^{\lfloor \hmax/h \rfloor}\frac{1}{m}\notag \\
        &\leq \sum_{h=1}^{\hmax}\frac{\kTree\hmax}{h}\left(1+ \log\left(\lfloor \hmax/h \rfloor\right)\right) \leq \kTree(1+\log(\hmax))\sum_{h=1}^{\hmax}\frac{\hmax}{h}\\
        & = \kTree\hmax (1+\log(\hmax))^2.\nonumber
    \end{align}
Finally, the number of calls to \(\OPT\) in the re-evaluation phase are bounded as follows:
\begin{align}
    \sum_{p=0}^{\lfloor\log_2(\hmax)\rfloor} \lfloor \hmax/2 \rfloor \leq \frac{\hmax}{2}(1+\log_2(\hmax)). 
\end{align}
To conclude, the total number of calls to the optimization solver \(\OPT\) is bounded by \(
    \kTree\hmax+ \kTree\hmax (1+\log(\hmax))^2+ \frac{\hmax}{2}(1+\log_2(\hmax)).\)
Since we set \(    \hmax= \left\lfloor \frac{2n}{5\kTree(1+\log(n))^2}\right\rfloor\), we conclude that 
\begin{align*}
    \kTree\hmax+ \kTree\hmax (1+\log(\hmax))^2+ \frac{\hmax}{2}(1+\log_2(\hmax)) \leq \frac{2n}{5} + \frac{2n}{5} + \frac{n}{5} = n. 
\end{align*}

\section{Sufficient conditions for Assumptions \ref{assm: OPT_Convergence_Convex}-\ref{assm: OPT_Convergence_StronglyConvex}}\label{app:SufficientConditionsAssm12}
In this section, we provide examples of the iterative convex optimization solver \(\textsf{OPT}\) so that Assumptions \ref{assm: OPT_Convergence_Convex}-\ref{assm: OPT_Convergence_StronglyConvex} hold. 
For this section, fix arbitrary \(\mathbf{w} = [\mathbf{y}_1, \mathbf{y}_2, ..., \mathbf{y}_{d_x+1}]\in W.\) 

\subsection{Sufficient condition for Assumption \ref{assm: OPT_Convergence_Convex}} For any $\mathbf{z} \in X$ and $s \in \mathbb{N}$, we define  \(
\textsf{OPT}(\mathbf{w}, \mathbf{z}, s) = \frac{1}{s} \sum_{j=1}^{s} \mathbf{x}_j,\)
where the iterates $\{\mathbf{x}_j\}_{j=1}^s$ are generated as
\begin{align*}
    \mathbf{x}_j &= \textsf{Proj}_{X}\bigl(\mathbf{x}_{j-1} - \eta \hat{h}_j\bigr), \quad 
    \hat{h}_j \in \partial_{\mathbf{x}} F(\mathbf{x}_{j-1}, \mathbf{w}), \quad 
    \mathbf{x}_1 = \mathbf{z}.
\end{align*}
Here, $\partial_{\mathbf{x}} F(\cdot, \mathbf{w})$ denotes the subdifferential of $F(\cdot, \mathbf{w})$ with respect to $\mathbf{x}$, and $\textsf{Proj}_{X}$ is the Euclidean projection onto the set $X$.

The following convergence guarantee holds.

\begin{proposition}[\cite{bubeck2015convex}]
Suppose that $X$ is contained in a Euclidean ball of radius $R$. Furthermore, assume that for every $\mathbf{x} \in X$, $\mathbf{w} \in W$, and $\hat{h} \in \partial_{\mbf{x}} F(\mathbf{x}, \mathbf{w})$, we have $\|\hat{h}\| \leq L$. If we set $\eta = \frac{RL}{\sqrt{s}}$, then Assumption~\ref{assm: OPT_Convergence_Convex} holds with $\Cpgd = RL$ and $\decayCpgd = 0.5$.
\end{proposition}

\subsection{Sufficient condition for Assumption \ref{assm: OPT_Convergence_StronglyConvex}}
For any $\mathbf{z} \in X$ and $s \in \mathbb{N}$, we define \(
\textsf{OPT}(\mathbf{w}, \mathbf{z}, s) = \mathbf{x}_s,\)
where the iterates $\{\mathbf{x}_j\}_{j=1}^s$ are generated as proximal-point iteration
\begin{align*}
    \mathbf{x}_j
    = \mathrm{prox}_{\eta F(\cdot,\mathbf{w})}(\mathbf{x}_{j-1}) :=
    \arg\min_{\mathbf{x}\in X}
    \left\{
        F(\mathbf{x},\mathbf{w})
        +
        \frac{1}{2\eta}\|\mathbf{x}-\mathbf{x}_{j-1}\|^2
    \right\},
    \qquad
    \mathbf{x}_1 = \mathbf{z}.
\end{align*}
There are several instances when we can evaluate the proximal map in closed form \cite{parikh2014proximal}.
The following convergence guarantee holds.

\begin{proposition}[\cite{ryu2014stochastic}]
Suppose that $X$ is contained in a Euclidean ball of radius $R$. Moreover, for every $\mbf{y} \in Y$, the function $f(\cdot,\mbf{y})$ is $\mu$-strongly convex on $X$ for some $\mu > 0$ and $L$-Lipschitz, uniformly over $Y$. Then, for any $\eta > 0$, the proximal-point iterates satisfy Assumption~\ref{assm: OPT_Convergence_StronglyConvex} with \(\SCpgd = 2RL\) and \( \decaySCpgd = \frac{1}{1+\eta\mu}.\)
\end{proposition}

\section{Proof of Main Results}\label{app: ProofTheoremConvex}
Let us introduce a few notations. First, we define \(\optDepth_{h,p}\) to be the largest depth of a node that satisfies the following three conditions: (i) it contains an optimal point, (ii) it has been evaluated more than \(2^p\) times by the optimization solver \(\OPT\), and (iii) the tree has been expanded up to depth \(h\). Second, for any \(h\in [\hmax],\) we define \(\optNode_h\in [\kTree^h]\) to be such that \(\optSol \in \setSOO_{h,\optNode_h}\), where \(\optSol\in\arg\max_{\mbf{w}\in \setSOO}\outerObj(\mbf{w})\). 
    
\subsection{Material Required to Prove Theorem \ref{thm: MainConvexResult}}\label{ssec: FirstThm}

\begin{lemma}\label{lem: OptDepthConvex}
    Suppose that Assumptions~\ref{assm: OPT_Convergence_Convex} and \ref{assm: Func_Part_Assm} hold. Then, for any pair \((h, p)\) with \(h \in [\hmax]\) and \(p \in \left[\left\lfloor \log_2(\hmax/h) \right\rfloor\right]\), the following implication holds: if  
\begin{align*}
    \constDecay \rateDecay^h \geq \Cpgd (2^p)^{-\decayCpgd}, \quad \text{and} \quad 
    \frac{\hmax}{2h \cdot 2^p} \geq C \rateDecay^{-d h}, \quad \text{then}~\optDepth_{h,p} = h.
\end{align*}
\end{lemma}
\begin{proof}
{The proof is inspired by the proof of \cite[Lemma 7]{bartlett2019simple}, which analyzes the case of noisy function evaluations. We adapt their proof to our setting, where each function evaluation is performed using \OPT.}

Fix a tuple \((h, p)\) that satisfies the conditions of this lemma. We prove by induction that \(\optDepth_{h',p} = h'\) for every \(h' \in [h]\).

\textbf{Base Case:} At depth \(1\), the initialization step in Algorithm~\ref{alg: Algorithm_ComputingSol} ensures that all child nodes are evaluated, so \(\optDepth_{1,p} = 1\).

\textbf{Inductive Step:} Fix an arbitrary \(h' \in [2:h]\) and assume the inductive hypothesis \(\optDepth_{h'-1,p} = h'-1\). We aim to show that \(\optDepth_{h',p} = h'\). By definition, the condition \(\optDepth_{h'-1, p} = h'-1\) ensures that the algorithm has expanded the node \((h'-1, i^\ast_{h'-1})\) with at least \(2^p\) evaluations. Therefore, it suffices to show that the algorithm also expands the node \((h', i^\ast_{h'})\) with at least \(2^p\) evaluations.

Let \(\bar{m}\) be the largest positive integer such that \(
2^p \leq \left\lfloor \frac{\hmax}{h' \bar{m}} \right\rfloor.\)
We claim that \(1 \leq \bar{m} \leq \left\lfloor \hmax/h' \right\rfloor\). First, since \(p \leq \lfloor \log_2(\hmax/h) \rfloor\) by design, it follows that \(2^p \leq \lfloor \hmax/h \rfloor\), hence such \(\bar{m} \geq 1\) exists. Moreover, since \(2^p \geq 1\), we must have \(\bar{m} \leq \lfloor \hmax/h' \rfloor\). Thus, the claimed bounds on \(\bar{m}\) hold.

This ensures that Algorithm~\ref{alg: Algorithm_ComputingSol} expands at least \(\bar{m}\) nodes at depth \(h'\) in \textsf{lines 3--10} of Algorithm \ref{alg: Algorithm_ComputingSol}. Let \(I_{\bar{m}}\) denote the set of the first \(\bar{m}\) nodes at depth \(h'\) that are evaluated in \textsf{lines 3--10} of Algorithm \ref{alg: Algorithm_ComputingSol}. By the definition of \(\bar{m}\), each node in \(I_{\bar{m}}\) is evaluated at least \(2^p\) times since \(
\left\lfloor \frac{\hmax}{h' \bar{m}} \right\rfloor \geq 2^p.\)
The proof is complete if we show that \((h', i^\ast_{h'}) \in I_{\bar{m}}\). 

Suppose for contradiction that \((h', i^\ast_{h'}) \notin I_{\bar{m}}\).
For any \((h',i)\in I_{\bar{m}},\) we claim that
\begin{align}\label{eq: Claim1}
\outerObj(\mathbf{w}_{h',i}) + \constDecay \rateDecay^{h'} \geq \estSOO_{h',i}.
\end{align}
Indeed, this holds because
\begin{align*}
\outerObj(\mathbf{w}_{h',i}) + \constDecay \rateDecay^{h'}
&\overset{(a)}{\geq} \outerObj(\mathbf{w}_{h',i}) + \constDecay \rateDecay^{h} \overset{(b)}{\geq} \outerObj(\mathbf{w}_{h',i}) + \Cpgd (2^p)^{-\decayCpgd}\\
& \overset{(c)}{\geq} \outerObj(\mathbf{w}_{h',i}) + \Cpgd (\num_{h',i})^{-\decayCpgd} \overset{(d)}{\geq} \estSOO_{h',i},
\end{align*}
where \((a)\) uses \(h' \leq h\) and \(\rateDecay \in (0,1)\), \((b)\) follows from the condition in the statement of Lemma \ref{lem: OptDepthConvex}, \((c)\) uses the fact that \(\num_{h',i} \geq 2^p\), and \((d)\) follows from Assumption~\ref{assm: OPT_Convergence_Convex}.
 
Next, we claim that
\begin{align}\label{eq: Claim2}
\estSOO_{h',\optNode_{h'}} \geq \outerObj(\optSol) - \constDecay \rateDecay^{h'}.
\end{align}
Indeed, \(
\estSOO_{h',\optNode_{h'}} = F(\OPT(\varSOO_{h',\optNode_{h'}},  \optSOO_{h',\optNode_{h'}}, \num_{h',\optNode_{h'}}),\varSOO_{h',\optNode_{h'}}) \overset{(a)}{\geq} \min_{{\mbf{x}\in X}} F({\mbf{x}}, \varSOO_{h',\optNode_{h'}}) = \outerObj(\mathbf{w}_{h',i^\ast_{h'}}) \overset{(b)}{\geq} \outerObj(\mbf{w}^\ast) - \constDecay \rateDecay^{h'},\)
where \((a)\) follows since \(\OPT\) provides a feasible solution (Assumption~\ref{assm: OPT_Convergence_Convex}), and \((b)\) follows from Assumption~\ref{assm: Func_Part_Assm}.
Combining \eqref{eq: Claim1} and \eqref{eq: Claim2}, we get:
\begin{align}\label{eq: Compare1}
\outerObj(\mathbf{w}_{h',i}) + \constDecay \rateDecay^{h'}
\overset{(a)}{\geq} \estSOO_{h',i}
\overset{(b)}{\geq} \estSOO_{h',\optNode_{h'}}
\overset{(c)}{\geq} \outerObj(\mbf{w}^\ast) - \constDecay \rateDecay^{h'},
\end{align}
where \((a)\) is from \eqref{eq: Claim1}, \((b)\) uses the assumption \((h',\optNode_{h'}) \notin I_{\bar{m}}\), and \((c)\) is from \eqref{eq: Claim2}.

From \eqref{eq: Compare1}, each \((h',i)\in I_{\bar{m}}\) is \(2\constDecay \rateDecay^{h'}\)-suboptimal. Since \((h',\optNode_{h'})\) is also \(2\constDecay \rateDecay^{h'}\)-suboptimal (in fact, only \(\constDecay \rateDecay^{h'}\)-suboptimal), it holds that  \(
\mathcal{N}_{h'}(2\constDecay \rateDecay^{h'}) \geq \bar{m} + 1.
\)
We now lower-bound \(\bar{m}\). By definition 
\[
2^p \geq \frac{\hmax}{h'(\bar{m} + 1)} \geq \frac{\hmax}{2h' \bar{m}} \quad \Rightarrow \quad \bar{m} \geq \frac{\hmax}{2h' 2^p}.
\]

To conclude,
\begin{align*}
\mathcal{N}_{h'}(2\constDecay \rateDecay^{h'}) 
\geq \bar{m} + 1 \geq \frac{\hmax}{2h2^p} + 1 
\overset{(a)}{\geq} C \rateDecay^{-d h} + 1 
\overset{(b)}{\geq} C \rateDecay^{-d h'} + 1,
\end{align*}
where  \((a)\) uses the lemma's assumption,
and \((b)\) uses \(h' \leq h\) and \(\rateDecay \in (0,1)\).

However, by the definition of near-optimality dimension (Definition~\ref{def: NearOptDim}), it must hold that \(
\mathcal{N}_{h'}(2\constDecay \rateDecay^{h'}) \leq C \rateDecay^{-d h'},
\)
which yields a contradiction. Therefore, we conclude that \((h', \optNode_{h'}) \in I_{\bar{m}}\) for every \(h' \in [h]\), and hence \(\optDepth_{h',p} = h'\).
\end{proof}

\subsubsection{Proof of Theorem \ref{thm: MainConvexResult}}
In order to bound the optimality gap, we claim that,
\begin{align}
    \regret(\iter) \leq \constDecay \rateDecay^{\optDepth_{\hmax,p} + 1} + \Cpgd \hmax^{-\decayCpgd} \quad \text{for all} \ p \in [\log_2(\hmax)]. 
\end{align}
Indeed, for every \(p \in [P]\),
{\small\begin{align}\label{eq: GapBound}
     &\regret(\iter) = \outerObj(\optSol) - \outerObj(\returnSol) \overset{(a)}{\leq} \outerObj(\optSol) + \Cpgd \left(\hmax/2\right)^{-\decayCpgd} - \returnVal \notag \\ 
     &\overset{(b)}{\leq} \outerObj(\optSol) + \Cpgd \left(\hmax/2\right)^{-\decayCpgd} - \estSOO_{h_p, i_p} \overset{(c)}{\leq} \outerObj(\optSol) + \Cpgd \left(\hmax/2\right)^{-\decayCpgd} - \estSOO_{\bar{h}(p), i^\ast_{\bar{h}(p)}} \notag \\ 
     &\overset{(c)}{=} \outerObj(\optSol) + \Cpgd \left(\hmax/2\right)^{-\decayCpgd} - F\left(\OPT\left(\varSOO_{\bar{h}(p), i^\ast_{\bar{h}(p)}}, \optSOO_{\bar{h}(p), i^\ast_{\bar{h}(p)}}, \num_{\bar{h}(p), i^\ast_{\bar{h}(p)}}\right),\varSOO_{\bar{h}(p), i^\ast_{\bar{h}(p)}}\right) \notag \\ 
     &\overset{(d)}{\leq} \outerObj(\optSol) + \Cpgd \left(\hmax/2\right)^{-\decayCpgd} - \outerObj(\varSOO_{\bar{h}(p), i^\ast_{\bar{h}(p)}}) \overset{(e)}{\leq} \Cpgd \left(\hmax/2\right)^{-\decayCpgd} + \constDecay \rateDecay^{\bar{h}(p)},
\end{align}}
where \(\bar{h}(p) := \optDepth_{\hmax, p}\). Here, \((a)\) follows from Assumption~\ref{assm: OPT_Convergence_Convex}, \((b)\) holds because \(\returnVal \geq \estSOO_{h_p, i_p}\) (see \textsf{line 16} in Algorithm~\ref{alg: Algorithm_ComputingSol}), 
\((c)\) follows from the definition of \(\estSOO_{\bar{h}(p), i^\ast_{\bar{h}(p)}}\) in \textsf{line 7} in Algorithm \ref{alg: Algorithm_ComputingSol}, \((d)\) uses the definition of \(\outerObj(\cdot)\), and \((e)\) uses Assumption~\ref{assm: Func_Part_Assm}.

From \eqref{eq: GapBound}, we observe that bounding the optimality gap reduces to lower bounding \(\max_{p \in [P]} \bar{h}(p)\). To do so, we aim to construct a tuple \((\hd, \pd) \in \mathbb{Z}^2\) that satisfies the conditions of Lemma~\ref{lem: OptDepthConvex}. This yields:
\begin{align}
    \max_{p \in [P]} \bar{h}(p) 
    \geq \bar{h}(\pd) = \optDepth_{\hmax, \pd} \overset{(a)}{\geq} \optDepth_{\hd, \pd} 
    \overset{(b)}{=} \hd,
\end{align}
where \((a)\) follows from the fact that \(\optDepth_{\cdot, p}\) is non-decreasing in its first argument for every fixed \(p\), and \((b)\) follows from Lemma~\ref{lem: OptDepthConvex}.
Combining this with \eqref{eq: GapBound}, we conclude:
\begin{align}\label{eq: RegretFinalBound}
    \regret(\iter) \leq \Cpgd \left(\hmax/2\right)^{-\decayCpgd} + \constDecay \rateDecay^{\hd}.
\end{align}

Thus, to achieve the lowest possible optimality gap, it suffices to construct a pair \((\hd, \pd) \in \mathbb{Z}^2\) with the largest possible value of \(\hd\) satisfying the conditions of Lemma~\ref{lem: OptDepthConvex}.

Towards that goal, we first construct a tuple \((\tilde{h}, \tilde{p}) \in \mathbb{R}^2\) such that 
\begin{align}
    \constDecay \rateDecay^{\tilde{h}} = 2^{\decayCpgd}\Cpgd \cdot 2^{-\tilde{p}\decayCpgd}, \quad \text{and} \quad 
    \hmax = 2\tilde{h}C \cdot 2^{\tilde{p}} \cdot \rateDecay^{-d\tilde{h}}. \label{eq: tp}
\end{align}

Rearranging these equations, we obtain:
\begin{align}\label{eq: thtp_exp}
    \tilde{h} = \frac{1}{\left(d+\frac{1}{\decayCpgd}\right)\log\left(\frac{1}{\rateDecay}\right)} \cdot \mathcal{W}\left( \frac{\hmax \left(d+\frac{1}{\decayCpgd}\right)\log\left(\frac{1}{\rateDecay}\right)}{4C\Cpgd^{1/\decayCpgd}\constDecay^{-1/\decayCpgd}} \right), \quad 
    \tilde{p} = \frac{1}{\decayCpgd}\log_2\left(\frac{2^{\decayCpgd}\Cpgd}{\constDecay\rateDecay^{\tilde{h}}}\right),
\end{align}
where \(\mathcal{W}(\cdot)\) is the Lambert \(W\)-function.

We now distinguish two cases based on whether condition \eqref{eq: ConditionnotSatisfied} is satisfied or not:
\paragraph{Case I (Condition~\eqref{eq: ConditionTheorem} holds)}
This corresponds to the regime where \(\constDecay \rateDecay^{\tilde{h}} \leq 2^{\decayCpgd} \Cpgd\), i.e., \(\tilde{p} \geq 0\). Set \(\hd = \lfloor \tilde{h} \rfloor\) and \(\pd = \lfloor \tilde{p} \rfloor\). We now verify that this choice of \((\hd, \pd)\) satisfies the requirements of Lemma~\ref{lem: OptDepthConvex}.

First, since the Lambert \(W\)- function is non-negative on \(\mathbb{R}_+\), we have \(\tilde{h} \geq 0\), and hence \(\hd \geq 0\). Moreover, using~\eqref{eq: tp}, we observe that \(\hmax = 2\tilde{h}C \cdot 2^{\tilde{p}} \rateDecay^{-d\tilde{h}} \geq 2\tilde{h} \geq 2\lfloor \tilde{h} \rfloor = 2\hd\), so that \(\hd \leq \hmax\).

Next, we claim that \(\pd \leq \lfloor \log_2(\hmax / \hd) \rfloor\). Indeed, since \(\hmax = 2\tilde{h}C \cdot 2^{\tilde{p}} \rateDecay^{-d\tilde{h}} \geq \tilde{h} \cdot 2^{\tilde{p}} \geq \hd \cdot 2^{\pd}\), the inequality follows.

We also verify that \(\constDecay \rateDecay^{\hd} \geq \Cpgd \cdot 2^{-\pd \decayCpgd}\). Since \(\hd \leq \tilde{h}\), we have \(\constDecay \rateDecay^{\hd} \geq \constDecay \rateDecay^{\tilde{h}} = \Cpgd \cdot 2^{-(\tilde{p} - 1)\decayCpgd} \geq \Cpgd \cdot 2^{-\pd \decayCpgd}\).

Finally, we verify that \(\frac{\hmax}{2\hd \cdot 2^{\pd}} \geq C \rateDecay^{-d\hd}\). From~\eqref{eq: tp}, we note that \(\frac{\hmax}{2\hd \cdot 2^{\pd}} \geq \frac{\hmax}{2\tilde{h} \cdot 2^{\tilde{p}}} = C \rateDecay^{-d\tilde{h}} \geq C \rateDecay^{-d\hd}\), where the last inequality follows from the monotonicity of \(\rateDecay^{-dh}\) in \(h\).

Thus, all conditions of Lemma~\ref{lem: OptDepthConvex} are verified. The desired bound in~\eqref{eq: GapCondition_satisfied} then follows by substituting \(\hd\) into~\eqref{eq: RegretFinalBound}.

\paragraph{Case II (Condition~\eqref{eq: ConditionTheorem} does not hold)}
This corresponds to the regime where \(\constDecay \rateDecay^{\tilde{h}} > 2^{\decayCpgd} \Cpgd\), i.e., \(\tilde{p} \leq 0\). Define \(\hat{h}\) by the relation \(\hmax = 2\hat{h}C \rateDecay^{-d\hat{h}}\), so that \(\hat{h} = \frac{1}{d \log(1/\rateDecay)} \cdot \mathcal{W}\left( \frac{d \log(1/\rateDecay) \cdot \hmax}{2C} \right)\). Set \(\pd = 0\) and \(\hd = \lfloor \hat{h} \rfloor\). We now verify that this choice of \((\hd, \pd)\) satisfies the conditions of Lemma~\ref{lem: OptDepthConvex}.

First, since the Lambert-\(W\) function is non-negative on \(\mathbb{R}_+\), we have \(\hat{h} \geq 0\), and thus \(\hd \geq 0\). Moreover, since \(\hmax = 2\hat{h} C \rateDecay^{-d\hat{h}} \geq 2\hat{h} \geq 2\lfloor \hat{h} \rfloor = 2\hd\), it follows that \(\hd \leq \hmax\).

Next, we verify that \(\constDecay \rateDecay^{\hd} \geq \Cpgd \cdot 2^{-\pd \decayCpgd}\). Since \(\pd = 0\), it suffices to show that \(\constDecay \rateDecay^{\hd} \geq \Cpgd\). Observe that \(\Cpgd \leq \constDecay \rateDecay^{\tilde{h}} \cdot 2^{-\decayCpgd} \leq \constDecay \rateDecay^{\tilde{h}} \leq \constDecay \rateDecay^{\hat{h}} \leq \constDecay \rateDecay^{\hd}\), where we have used the fact that \(\hat{h} \leq \tilde{h}\). To verify this, note that
\[
2\hat{h} C \rateDecay^{-d\hat{h}} = \hmax = 2\tilde{h} C \rateDecay^{-d\tilde{h}} \left( \frac{2^{\decayCpgd} \Cpgd}{\constDecay \rateDecay^{\tilde{h}}} \right)^{1/\decayCpgd} \leq 2\tilde{h} C \rateDecay^{-d\tilde{h}},
\]
and since the function \(x \mapsto 2Cx \rateDecay^{-dx}\) is increasing on \(\mathbb{R}_+\), it follows that \(\hat{h} \leq \tilde{h}\).

Next, we verify that \(\frac{\hmax}{2\hd \cdot 2^{\pd}} \geq C \rateDecay^{-d\hd}\). Since \(\pd = 0\), this reduces to \(\frac{\hmax}{2\hd} \geq C \rateDecay^{-d\hd}\), which follows from \(\frac{\hmax}{2\hd} \geq \frac{\hmax}{2\hat{h}} = C \rateDecay^{-d\hat{h}} \geq C \rateDecay^{-d\hd}\).
Thus, all conditions of Lemma~\ref{lem: OptDepthConvex} are satisfied. 

The desired bound in~\eqref{eq: ConditionnotSatisfied} then follows by substituting \(\hd\) into~\eqref{eq: RegretFinalBound} and noting that \(2^{\decayCpgd} \Cpgd \hmax^{-\decayCpgd} \leq \constDecay \rateDecay^{\hd}\). Indeed, we have \(\Cpgd (\hmax/2)^{-\decayCpgd} \leq \constDecay \rateDecay^{\tilde{h}} \cdot \hmax^{-\decayCpgd} \leq \constDecay \rateDecay^{\tilde{h}} \leq \constDecay \rateDecay^{\hd}\), where the first inequality uses the assumption that \(2^{\decayCpgd} \Cpgd < \constDecay \rateDecay^{\tilde{h}}\). This concludes the proof.

\subsection{Material Required to Prove Theorem \ref{thm: MainStronglyConvexResult}}\label{ssec: SecondThm}
\begin{lemma}\label{lem: OptDepth_StronglyConvex}
    Suppose that Assumptions~\ref{assm: OPT_Convergence_StronglyConvex} and ~\ref{assm: Func_Part_Assm} hold. Then, for any pair \((h, p)\) with \(h \in [\hmax]\) and \(p \in \left[\left\lfloor \log_2(\hmax/h) \right\rfloor\right]\), the following implication holds: if
    \begin{align*}
        \constDecay \rateDecay^h \geq \SCpgd\decaySCpgd^{2^p},\quad  \text{and} \quad \hmax/(2h2^p)  \geq C\rateDecay^{-dh},
    \end{align*}
    then \(\optDepth_{h,p} = h.\)
\end{lemma}
\begin{proof}
    The proof is similar to that of Lemma \ref{lem: OptDepthConvex}. 
Fix a tuple \((h, p)\) that satisfies the conditions of this lemma. We prove by induction that \(\optDepth_{h',p} = h'\) for every \(h' \in [h]\).

\textbf{Base Case:} This is similar as in proof of Lemma \ref{lem: OptDepthConvex}.

\textbf{Inductive Step:} Fix an arbitrary \(h' \in [2:h]\) and assume the inductive hypothesis \(\optDepth_{h'-1,p} = h'-1\). Similar to Lemma \ref{lem: OptDepthConvex}, it suffices to show that the algorithm also expands the node \((h', i^\ast_{h'})\) with at least \(2^p\) evaluations.


Let \(\bar{m}\) be the largest positive integer such that \(
2^p \leq \left\lfloor \frac{\hmax}{h' \bar{m}} \right\rfloor.\)
We note that \(1 \leq \bar{m} \leq \left\lfloor \hmax/h' \right\rfloor\). 
This ensures that Algorithm~\ref{alg: Algorithm_ComputingSol} expands at least \(\bar{m}\) nodes at depth \(h'\) in \textsf{lines 3--10} of Algorithm \ref{alg: Algorithm_ComputingSol}. Let \(I_{\bar{m}}\) denote the set of the first \(\bar{m}\) nodes at depth \(h'\) that are evaluated in \textsf{lines 3--10} of Algorithm \ref{alg: Algorithm_ComputingSol}. By the definition of \(\bar{m}\), each node in \(I_{\bar{m}}\) is evaluated at least \(2^p\) times.
The proof will be complete if we show that \((h', i^\ast_{h'}) \in I_{\bar{m}}\). Suppose for contradiction that \((h', i^\ast_{h'}) \notin I_{\bar{m}}\).
For any \((h',i)\in I_{\bar{m}},\) we claim that
\begin{align}\label{eq: Claim1SC}
\outerObj(\mathbf{w}_{h',i}) + \constDecay \rateDecay^{h'} \geq \estSOO_{h',i}.
\end{align}
Indeed, this holds because
\begin{align*}
\outerObj(\mathbf{w}_{h',i}) + \constDecay \rateDecay^{h'}
&\overset{(a)}{\geq} \outerObj(\mathbf{w}_{h',i}) + \constDecay \rateDecay^{h} \overset{(b)}{\geq} \outerObj(\mathbf{w}_{h',i}) + \SCpgd\decaySCpgd^{2^p} \\
& \overset{(c)}{\geq} \outerObj(\mathbf{w}_{h',i}) + \SCpgd\decaySCpgd^{\num_{h',i}}\overset{(d)}{\geq} \estSOO_{h',i},
\end{align*}
where \((a)\) uses \(h' \leq h\) and \(\rateDecay \in (0,1)\), \((b)\) follows from the hypothesis of Lemma \ref{lem: OptDepth_StronglyConvex}, \((c)\) uses the fact that \(\num_{h',i} \geq 2^p\), and \((d)\) follows from Assumption~\ref{assm: OPT_Convergence_StronglyConvex}.
Next, we claim that
\begin{align}\label{eq: Claim2SC}
\estSOO_{h',\optNode_{h'}} \geq \outerObj(\optSol) - \constDecay \rateDecay^{h'}.
\end{align}
This follows similar to \eqref{eq: Claim2}. 
Combining \eqref{eq: Claim1SC} and \eqref{eq: Claim2SC}, we get:
\begin{align}\label{eq: Compare1SC}
\outerObj(\mathbf{w}_{h',i}) + \constDecay \rateDecay^{h'}
\overset{(a)}{\geq} \estSOO_{h',i}
\overset{(b)}{\geq} \estSOO_{h',\optNode_{h'}}
\overset{(c)}{\geq} \outerObj(\mbf{w}^\ast) - \constDecay \rateDecay^{h'}.
\end{align}

From \eqref{eq: Compare1SC}, each \((h',i)\in I_{\bar{m}}\) is \(2\constDecay \rateDecay^{h'}\)-suboptimal. Since \((h',\optNode_{h'})\) is also \(2\constDecay \rateDecay^{h'}\)-suboptimal (in fact, only \(\constDecay \rateDecay^{h'}\)-suboptimal), we have:
\begin{align*}
\mathcal{N}_{h'}(2\constDecay \rateDecay^{h'}) 
\geq \bar{m} + 1 \geq \frac{\hmax}{2h2^p} + 1 
\overset{(a)}{\geq} C \rateDecay^{-d h} + 1 
\overset{(b)}{\geq} C \rateDecay^{-d h'} + 1,
\end{align*}
which follows similar to the proof of Lemma \ref{lem: OptDepthConvex}. 
However, by the definition of near-optimality dimension (Definition~\ref{def: NearOptDim}), we have \(
\mathcal{N}_{h'}(2\constDecay \rateDecay^{h'}) \leq C \rateDecay^{-d h'},\)
which yields a contradiction. Therefore, we conclude that \((h', \optNode_{h'}) \in I_{\bar{m}}\) for every \(h' \in [h]\), and hence \(\optDepth_{h',p} = h'\).
    
\end{proof}

\subsubsection{Proof of Theorem \ref{thm: MainStronglyConvexResult}}
In order to bound the optimality gap, we claim that
\begin{align}
    \regret(\iter) \leq \constDecay\rateDecay^{\optDepth_{\hmax,p}+1} + \SCpgd\decaySCpgd^{\hmax/2} \quad \text{for all} \ p \in \left[P\right]. 
\end{align}
This follows analogous to \eqref{eq: GapBound}, Assumption \ref{assm: OPT_Convergence_StronglyConvex} in place of Assumption \ref{assm: OPT_Convergence_Convex}. 

We note that to obtain the desired optimality gap bound, it is sufficient to lower bound \(\max_{p\in [P]}\bar{h}(p)\). In order to lower bound this term, we hope to find a tuple \((\hd, \pd)\in \N^2\) such that the conditions posed in Lemma \ref{lem: OptDepth_StronglyConvex} are satisfied. Consequently, we can lower bound \(
    \max_{p\in [P]}\optDepth_{\hmax,p} \geq \optDepth_{\hmax, \pd} \underset{(a)}{\geq} \optDepth_{\hd,\pd} \underset{(b)}{=}{\hd}, \)
where \((a)\) is due  to the fact that \(\optDepth_{\cdot, p}\) is a non-decreasing function, for every \(p\), and \((b)\) is due to Lemma \ref{lem: OptDepth_StronglyConvex}. Consequently, have
\begin{align}\label{eq: RegretFinalBoundSC}
    \regret(\iter) \leq \SCpgd\decaySCpgd^{\hmax/2} + \constDecay\rateDecay^{h^\dagger}. 
\end{align}
 Thus, to obtain a lower optimality gap we must construct \((\hd, \pd)\in \N^2\) with largest values of \(\hd\) that satisfy the conditions in Lemma \ref{lem: OptDepth_StronglyConvex}. 
Towards this goal, we first construct a tuple \((\tilh, \tp)\in \mathbb{R}^2\) that satisfy the following conditions:
\begin{align}\label{eq: thSC}
    \constDecay \rateDecay^{\tilde{h}} = \SCpgd\decaySCpgd^{2^{\tp-1}}, \quad  \text{and} \quad 
    \hmax = 2\tilde{h}C2^{\tilde{p}}\rateDecay^{-d\tilde{h}}.
\end{align}
Moving forward, without loss of generality, we assume that \(\SCpgd = \constDecay.\) This can always be done by selecting the maximum of these two values. Using the definition of Lambert W-function in \eqref{eq: thSC}, we obtain 
\begin{align}\label{eq: thtp_expSC}
    \tilh = \frac{1}{\frac{d}{2}\log(1/\rateDecay)} \mathcal{W}\left( \frac{d}{2}\log(1/\rateDecay)\sqrt{\frac{\hmax}{4C}} \right), \quad 
    \tp = \log_2\left(2\tilh\frac{\log(1/\rateDecay)}{\log(1/\decaySCpgd)}\right),
\end{align}
where \(\mathcal{W}(\cdot)\) is the Lambert W-function.
We now distinguish two cases based on whether condition \eqref{eq: ConditionTheoremSC} is satisfied or not:

\paragraph{Case I (Condition~\eqref{eq: ConditionTheoremSC} holds)} This corresponds to the regime when  \(\tp \geq 0\), or equivalently, \(\tilh \geq \frac{\log(1/\decaySCpgd)}{2\log(1/\rateDecay)}\). Set \(\hd = \lfloor \tilh \rfloor\) and \(\pd = \lfloor \tp \rfloor\). We now verify that this choice of \((\hd, \pd)\) satisfies the requirements of Lemma~\ref{lem: OptDepth_StronglyConvex}.

{
Similar to the Case I in the proof of Theorem \ref{thm: MainConvexResult}, we can show that \(0 \leq \hd \leq \hmax,\) and \(\pd \leq \lfloor \log_2(\hmax / \hd) \rfloor\).
We also verify that \(\constDecay \rateDecay^{\hd} \geq \Cpgd 2^{-\pd \decayCpgd}\). Since we assumed that \(\Cpgd = \constDecay,\) it suffices to show that \(\rateDecay^{\hd} \geq 2^{-\pd \decayCpgd}\). Using \eqref{eq: thSC} and the fact that \(\hd \leq \tilde{h}\), we obtain \(\rateDecay^{\hd} \geq \rateDecay^{\tilh} = \decaySCpgd^{2^{\tp-1}} \geq \decaySCpgd^{2^{\pd}}\).
Finally, we verify that \(\frac{\hmax}{2\hd \cdot 2^{\pd}} \geq C \rateDecay^{-d\hd}\). From~\eqref{eq: thSC}, we note that \(\frac{\hmax}{2\hd 2^{\pd}} \geq \frac{\hmax}{2\tilh 2^{\tp}} = C\rateDecay^{-d\tilh} \geq C\rateDecay^{-d\hd}\), where the last inequality follows from the monotonicity of \(\rateDecay^{-dh}\) in \(h\).
Thus, all conditions of Lemma~\ref{lem: OptDepth_StronglyConvex} are verified. The desired bound in~\eqref{eq: GapCondition_satisfiedSC} then follows by substituting \(\hd\) into~\eqref{eq: RegretFinalBoundSC}.
}



\paragraph{Case II (Condition~\eqref{eq: ConditionTheoremSC} does not hold)} 
{This corresponds to the regime where \(\tilh < \frac{\log(1/\decaySCpgd)}{2\log(1/\rateDecay)}\), i.e., \(\tp < 0\). Define \(\hat{h}\) by the relation \(\frac{\hmax}{2\hat{h}} = C\rateDecay^{-d\hat{h}}\), so that \(\hat{h} = \frac{1}{d\log(1/\rateDecay)} \mathcal{W}\left( \frac{d\log(1/\rateDecay)\hmax}{2C} \right)\). Set \(\pd = 0\) and \(\hd = \lfloor \hat{h} \rfloor\). We now verify that this choice of \((\hd, \pd)\) satisfies the conditions of Lemma~\ref{lem: OptDepth_StronglyConvex}.

Similar to the Case II in the proof of Theorem \ref{thm: MainConvexResult}, we can show that \(0 \leq \hd \leq \hmax,\) and \(\pd \leq \lfloor \log_2(\hmax / \hd) \rfloor\).
Next, we verify that \(\constDecay \rateDecay^{\hd} \geq \SCpgd\decaySCpgd^{2^{\pd}}\). Since we have assumed that \(\SCpgd = \constDecay\) and have set \(\pd = 0,\) it suffices to show that \(\rateDecay^{\hd} \geq \decaySCpgd\). Observe that \(\decaySCpgd \leq \rateDecay^{2\tilh} \leq \constDecay\rateDecay^{\tilh} \leq \constDecay\rateDecay^{\hat{h}} \leq \constDecay\rateDecay^{\hd}\), where we have used the fact that \(\hat{h} \leq \tilde{h}\). To verify this, note that
\[
2\hat{h} C \rateDecay^{-d\hat{h}} = \hmax \underset{\eqref{eq: thSC}}{=} 2\tilh C \rateDecay^{-d\tilh} \cdot \left(2\tilh \frac{\log(1/\rateDecay)}{\log(1/\decaySCpgd)}\right) \leq 2\tilde{h} C \rateDecay^{-d\tilde{h}},
\]
where the last inequality holds because \(2\tilh \frac{\log(1/\rateDecay)}{\log(1/\decaySCpgd)} \leq 1\) in the regime of Case II.
Since the function \(x \mapsto 2Cx \rateDecay^{-dx}\) is increasing on \(\mathbb{R}_+\), it follows that \(\hat{h} \leq \tilde{h}\).
Next, we verify that \(\SCpgd\decaySCpgd^{\hmax/2} \leq \constDecay\rateDecay^{\hd}\). Since \(\pd = 0\), this reduces to \(\frac{\hmax}{2\hd} \geq C \rateDecay^{-d\hd}\), which follows from \(\frac{\hmax}{2\hd} \geq \frac{\hmax}{2\hat{h}} = C \rateDecay^{-d\hat{h}} \geq C \rateDecay^{-d\hd}\).
Thus, all conditions of Lemma~\ref{lem: OptDepth_StronglyConvex} are satisfied.
The desired bound in~\eqref{eq: ConditionnotSatisfiedSC} then follows by substituting \(\hd\) into~\eqref{eq: RegretFinalBoundSC} and noting that \(\SCpgd\decaySCpgd^{\hmax/2}  \leq \constDecay\rateDecay^{\hd}.\) Indeed, \(\SCpgd\decaySCpgd^{\hmax/2} \leq \constDecay\rateDecay^{\tilh \hmax} \leq \constDecay\rateDecay^{\tilh} \leq \constDecay\rateDecay^{\hd},\)  where the first inequality follows when condition \eqref{eq: ConditionTheoremSC} does not hold as  \(\decaySCpgd < \rateDecay^{2\tilh}\). This concludes the proof.

}

\section{Preliminaries on Lambert W-function}\label{app: LambertW}
Consider the equation \( y = x\exp(x) \), where \( x, y \in \mathbb{R} \). The solution to this equation is given by \( x = \lamW(y) \), where \(\lamW\) denotes the Lambert \(W\) function. By definition, this function satisfies \( y = \lamW(y)\exp(\lamW(y)) \).

\begin{lemma}\label{lem: W_properties}
Let \(\lamW(\cdot)\) denote the Lambert \(W\) function. Then:
\begin{enumerate}[labelwidth=*,align=left, widest=iii, leftmargin=*]
    \item[(i)] If \( y \geq 0 \), then \( \lamW(y) \geq 0 \).
    \item[(ii)] \( \lamW(\cdot) \) is a strictly increasing function on the domain \( [0, +\infty [ \).
    \item[(iii)] If \( y > e \), then \( \lamW(y) \geq \log\left({y} / {\log(y)}\right) \).
    \item[(iv)] {\color{blue}\(\lim_{x\downarrow 0}\lamW(x)/x = 1\)}
\end{enumerate}
\end{lemma}

\section{Proof of Proposition \ref{prop: SufficientConditins}}\label{app: SufficientConditions}
In this section, we provide a proof of Proposition \ref{prop: SufficientConditins}. Before that, we present an intermediate result that will be used in the proof.

\begin{lemma}\label{lem: LipschitzG}
Under Assumption \ref{assm: BigX}, the function \(\outerObj(\cdot)\) is \(L_G-\)Lipschitz, where \(L_G=\sqrt{d_x+1}L_f\). 
\end{lemma}

\begin{proof}
First, we show that under Assumption \ref{assm: BigX}, the mapping \(W \ni \mbf{w}\mapsto F(\mbf{x},\mbf{w})\) is Lipschitz. Fix \(\mbf{x}\in X\). For any \(\mbf{w} = (\mbf{y}_i)_{i\in [d_x+1]}, \bar{\mbf{w}} = (\bar{\mbf{y}}_i)_{i\in [d_x+1]}\in W,\)
\begin{align*}
&|F(\mbf{x},\mbf{w})-F(\mbf{x},\bar{\mbf{w}})|
=
\left|
\max_{i\in[d_x+1]} f(\mbf{x},\mbf{y}_i)
-
\max_{i\in[d_x+1]} f(\mbf{x},\bar{\mbf{y}}_i)
\right| \\
&\le
\max_{i\in[d_x+1]}
|f(\mbf{x},\mbf{y}_i)-f(\mbf{x},\bar{\mbf{y}}_i)| \le
L_f \max_{i\in[d_x+1]} \|\mbf{y}_i-\bar{\mbf{y}}_i\|.
\end{align*}
Hence, for every \(\mbf{x}\in X\), \(
F(\mbf{x},\mbf{w})
\le
F(\mbf{x},\tilde{\mbf{w}})
+
L_f \max_{i\in[d_x+1]} \|\mbf{y}_i-\tilde{\mbf{y}}_i\|.\)
Taking infimum over \(\mbf{x}\in X\) on both sides gives
\[
G(\mbf{w})
=
\inf_{\mbf{x}\in X} F(\mbf{x},\mbf{w})
\le
\inf_{\mbf{x}\in X}
\left(
F(\mbf{x},\tilde{\mbf{w}})
+
L_f \max_{i\in[d_x+1]} \|\mbf{y}_i-\tilde{\mbf{y}}_i\|
\right).
\]
Since the second term is independent of \(\mbf{x}\), this becomes
\[
G(\mbf{w})
\le
G(\tilde{\mbf{w}})
+
L_f \max_{i\in[d_x+1]} \|\mbf{y}_i-\tilde{\mbf{y}}_i\|.
\]
Interchanging the roles of \(\mbf{w}\) and \(\tilde{\mbf{w}}\), we similarly obtain
\[
G(\tilde{\mbf{w}})
\le
G(\mbf{w})
+
L_f \max_{i\in[d_x+1]} \|\mbf{y}_i-\tilde{\mbf{y}}_i\|.
\]
Combining the two inequalities yields \(
|G(\mbf{w})-G(\tilde{\mbf{w}})|
\le
L_f \max_{i\in[d_x+1]} \|\mbf{y}_i-\tilde{\mbf{y}}_i\|,\)
which proves the claim.
\end{proof}

\begin{proof}[Proof of Proposition \ref{prop: SufficientConditins}]
From Lemma \ref{lem: LipschitzG}, we know that under Assumption \ref{assm: BigX}, \(G(\cdot)\) is Lipschitz. Therefore, for any \(h\in [h_{\max}],\) \(|G(\mbf{w}-G(\mbf{w}^\ast))|\leq L_G\|\mbf{w}-\mbf{w}^\ast\|\) for any \(\mbf{w}\in W_{h,i_h^\ast}\), where \(i_h^\ast\) denotes the index such that \(\mbf{w}^\ast\in W_{h,i_h^\ast}.\) By the geometric condition \eqref{eq: ConditionGeometric} the diameter of \(W_{h,i^\ast_h}\) is atmost \(\alpha\beta^h,\) which implies that the function variation over this cell is bounded by \(L_G\alpha\beta^h.\) This satisfies the requirement in Assumption \ref{assm: Func_Part_Assm}.
\end{proof}
\section{Proof of Proposition \ref{prop: HandCraftExample}}\label{app: proof_handcraft}
To prove Proposition~4.1, we begin by reformulating~(4.1) by introducing the aggregate variables \(
    X := \mathbf{1}^\top \mathbf{x},
    Y := \mathbf{1}^\top \mathbf{y}.\)
Under this transformation, (4.1) reduces to the one-dimensional form
\begin{align*}
    \min_{X \in [-c \xSetDim, \; c \xSetDim]} \;
    \max_{Y \in [-\ySetDim, \; \ySetDim]} \;
    \hat{f}(X, Y)
    = -Y^{3} + XY .
\end{align*}
For any fixed \(X \in [-c \xSetDim, c \xSetDim]\), define 
\(    Y^\ast(X) := \underset{Y \in [-\ySetDim, \, \ySetDim]}{\arg\max} \hat{f}(X, Y).
\)
We claim that \begin{align*} Y^\ast(X) = \begin{cases} -\ySetDim & \text{if } X \in [-c\xSetDim,\, 0.75 \ySetDim^2[, \\ \{-\ySetDim,\, 0.5 \ySetDim\} & \text{if } X = 0.75 \ySetDim^2, \\ \sqrt{X/3} & \text{if } X \in\, ]0.75 \ySetDim^2,\, c\xSetDim[. 
\end{cases} 
\end{align*}
Since the mapping \(Y \mapsto \hat{f}(X,Y)\) is cubic in \(Y\), the maximizer must lie either at the boundary points or at a stationary point. That is, for any \(X \in [-c\xSetDim, c\xSetDim],\)
\[
    Y^\ast(X) \subseteq \big\{-\ySetDim, \, \ySetDim, \, \sqrt{X/3}, \, -\sqrt{X/3}\big\}.
\]
Define the mappings
\begin{align*} [-c\xSetDim, c\xSetDim] \ni X \mapsto F_1(X) &= \hat{f}(X, -\ySetDim) = \ySetDim^3 - \ySetDim X, \\ [-c\xSetDim, c\xSetDim] \ni X \mapsto F_2(X) &= \hat{f}(X, \ySetDim) = -\ySetDim^3 + \ySetDim X, \\ [0, c\xSetDim] \ni X \mapsto F_3(X) &= \hat{f}\left(X, \sqrt{\frac{X}{3}}\right) = \frac{2}{3\sqrt{3}}X^{3/2}, \\ [0, c\xSetDim] \ni X \mapsto F_4(X) &= \hat{f}\left(X, -\sqrt{\frac{X}{3}}\right) = -\frac{2}{3\sqrt{3}}X^{3/2}. \end{align*}
Note that \(F_1\) and \(F_4\) are strictly decreasing, whereas \(F_2\) and \(F_3\) are strictly increasing.

To determine \(Y^\ast(X)\), find the maximum value of  
\(\{F_1(X), F_2(X), F_3(X), F_4(X)\}\). We consider the following cases to find the best candidate value: 

\begin{itemize}
    \item \textbf{Case I (\(X \le 0\)):} Only \(F_1\) and \(F_2\) are defined, with \(F_1(X) \geq F_2(X)\). Hence \(Y^\ast(X) = -\ySetDim\).
    
    \item \textbf{Case II (\(0 < X < 0.75\ySetDim^2\)):} One can verify that 
    \begin{align*}
    &F_1(X) > 0.25\ySetDim^3,~ F_2(X) < -0.25\ySetDim^3, \\  &F_3(X) < 0.25 \ySetDim^3,~  -0.25 \ySetDim^3 < F_4(X) < 0. \end{align*}
    Therefore, 
\(        F_1(X) > \max\{F_2(X), F_3(X), F_4(X)\},\)
    so again \(Y^\ast(X) = -\ySetDim\).

    \item  \textbf{Case III (\(0.75\ySetDim^2 = X\)):} Using the similar analysis as Case II, we conclude that \(F_1(X) = F_3(X) > \max\{F_2(X), F_4(X)\}.\) Therefore, \(Y^\ast(X) \in \{ -\ySetDim, \sqrt{X/3}\} = \{ -\ySetDim, 0.5\ySetDim\}.\)
    
    \item \textbf{Case IV (\(0.75\ySetDim^2 < X < 3\ySetDim^2\)):} We note that 
    \begin{align*} 
    &F_1(X) < 0.25\ySetDim^3,~ -0.25\ySetDim^3 < F_2(X) < 2\ySetDim^2,\\ & 0.25\ySetDim^3 < F_3(X) < 2\ySetDim^3, ~ F_4(X) < -0.25 \ySetDim^3. \end{align*}
    Thus, \(\max\{F_2(X), F_3(X)\} > \max\{F_1(X), F_4(X)\}.\) Moreover, using the fact that \(F_2\) and \(F_3\) are strictly increasing with \(F_2(3\ySetDim^2) = F_3(3\ySetDim^2)\), we conclude that  \(F_3(X) > \max\{F_1(X), F_2(X), F_4(X)\}.\) It follows that \(Y^\ast(X) = \sqrt{X/3}\).
    
    \item \textbf{Case V (\( X = 3\ySetDim^2\)):} Using the similar analysis as Case IV, we conclude that \(F_2(X) = F_3(X) > \max\{F_1(X), F_4(X)\}.\)  Since \(\sqrt{X/3}= \ySetDim\), we conclude that \(Y^\ast(X) = \ySetDim.\)

    \item \textbf{Case VI (\(X > 3 \ySetDim^2\)):} We note that \begin{align*} 
    &F_1(X) < -2\ySetDim^3, ~ F_2(X) > 2\ySetDim^2,\\ & F_3(X) > 2\ySetDim^3, ~ F_4(X) < -2 \ySetDim^3. 
    \end{align*} Moreover, we note that \(F_2(3\ySetDim^2) = F_3(3\ySetDim^2)\) and \begin{align*} \frac{d}{dX}F_2(X) = \ySetDim < \frac{d}{dX}F_3(X) = \sqrt{\frac{X}{3}}. \end{align*} Therefore, \(Y^\ast(X) = \sqrt{X/3}\).
    \end{itemize}

Finally, we define the function
\[
    \Phi(X) := \max_{Y \in [-\ySetDim, \ySetDim]} \hat{f}(X, Y)
    = \hat{f}(X, Y^\ast(X)).
\]
From the above analysis,
\[
    \Phi(X) =
    \begin{cases}
        \ySetDim^3 - \ySetDim X, & X \in [-c\xSetDim,\, 0.75 \ySetDim^2], \\[0.5ex]
        \tfrac{2}{3\sqrt{3}} X^{3/2}, & X \in \,]0.75 \ySetDim^2,\, c\xSetDim].
    \end{cases}
\]
We note that \(\Phi\) is decreasing on \([-c\xSetDim,\, 0.75\ySetDim^2]\) and increasing on \([0.75\ySetDim^2,\, c\xSetDim]\). Hence its global minimum is attained at \(X = 0.75\ySetDim^2\). This completes the proof.

\section{Additional Details on Numerical Experiments}\label{app:NumericalExtras}
{In this section, we provide the implementation details for the numerical experiments reported in
Section~\ref{sec: Numerics}. We first recall the AGP algorithm from
\cite{xu2023unified}, of which GDA is a special case. We then specify the
parameter choices used for AGP and GDA in each class of experiments.

The AGP algorithm is a gradient-based method proposed in \cite{xu2023unified}
for solving
\[
    \min_{\mbf{x}\in X}\max_{\mbf{y}\in Y} f(\mbf{x},\mbf{y}).
\]
For every \(t\in[1:T]\), the AGP iterates are given by
\begin{align*}
\mbf{x}_{t+1}
&=
\Pi_X\!\left(
\mbf{x}_t
-
\frac{\nabla_{\mbf{x}} f(\mbf{x}_t,\mbf{y}_t)}{\beta(t+1)}
-
\frac{b(t+1)}{\beta(t+1)}\,\mbf{x}_t
\right), \\
\mbf{y}_{t+1}
&=
\Pi_Y\!\left(
\mbf{y}_t
+
\frac{\nabla_{\mbf{y}} f(\mbf{x}_t,\mbf{y}_t)}{\gamma(t+1)}
-
\frac{c(t+1)}{\gamma(t+1)}\,\mbf{y}_t
\right),
\end{align*}
where \(\Pi_X\) and \(\Pi_Y\) denote the Euclidean projection operators onto
\(X\) and \(Y\), respectively. The functions
\(b(\cdot)\), \(c(\cdot)\), \(\beta(\cdot)\), and \(\gamma(\cdot)\) denote
the algorithmic schedules. GDA is obtained as the special case of AGP with
\(b(\cdot)=0\) and \(c(\cdot)=0\).

Following \cite{xu2023unified}, we choose the step-size schedules based on some hyper-parameter tuning and
the Lipschitz constants of the partial gradients of \(f\). We now describe
the exact schedules used for AGP and GDA in each experiment.

\subsection{Parameters of AGP and GDA for the Hand-Crafted Example in Section~\ref{subsec:hand_crafted_example}}
\label{ref:app_AGP_GDA_parameters_TE2}

For the AGP implementation, for every \(t\in[1:T]\), we use
\begin{align*}
b(t)
&=
\frac{0.95}{\bar\eta}\,t^{-1/4}, \\
c(t)
&=
0, \\
\beta(t)
&=
\frac{1}{\bar\eta}, \\
\gamma(t)
&=
\bar\nu+\bar\gamma(t), \\
\bar\gamma(t)
&=
\bar\eta d_y
+
\frac{16 d_y}{\bar\eta\,b(t+1)^2}
\max\!\left(
\frac{0.1128}{\bar\eta}
\cdot
\frac{6d_y^{3/2}+\bar\nu-\bar\eta d_y}{d_y},
\,2
\right)
-
2\bar\nu,
\end{align*}
where we set \(\bar\eta=10\) and \(\bar\nu=0.1\).

For the GDA implementation, we use
\begin{align*}
\beta(t)
&=
\frac{1}{0.9}(t+1)^{0.6}, \\
\gamma(t)
&=
\frac{1}{0.6}(t+1)^{0.9}.
\end{align*}

\subsection{Parameters of AGP and GDA for the Benchmark Problems in Section~\ref{ssec:BenchmarkSIPAMPL}}

For the GDA implementation, we use the schedules
\[
\beta(t)=600\,(t+1)^{0.9},
\qquad
\gamma(t)=900\,(t+1)^{0.6},
\]
for all benchmark problems considered in Section~\ref{ssec:BenchmarkSIPAMPL}.

For the AGP implementation, we use
\begin{align*}
b(t)
&=
\frac{0.95}{\bar\eta}\,t^{-1/4}, \\
c(t)
&=
0, \\
\beta(t)
&=
\frac{1}{\bar\eta}, \\
\gamma(t)
&=
\bar\eta L_{21}^2
+
\frac{16\tau L_{21}^2}{\bar\eta\,b(t+1)^2}
-
\bar\nu,
\end{align*}
where
\[
\tau
=
\max\!\left(
\frac{0.1128}{\bar\eta}
\cdot
\frac{6L_{21}^3+\bar\nu-\bar\eta L_{21}^2}{L_{21}^2},
\,2
\right).
\]
Here, \(\bar\eta\) and \(\bar\nu\) are problem-dependent hyperparameters.
For every \(\mbf{x}\in X\) and \(\mbf{y}\in Y\), let \(L_{11}\) denote the
Lipschitz constant of \(\nabla_{\mbf{x}} f(\cdot,\mbf{y})\), \(L_{21}\) denote
the Lipschitz constant of \(\nabla_{\mbf{x}} f(\mbf{x},\cdot)\), \(L_{22}\)
denote the Lipschitz constant of \(\nabla_{\mbf{y}} f(\mbf{x},\cdot)\), and
\(L_{12}\) denote the Lipschitz constant of
\(\nabla_{\mbf{y}} f(\cdot,\mbf{y})\).

The hyperparameters used for each benchmark problem are reported in
Table~\ref{tab:agp_gda_hyperparameters}.

\begin{table}[h]
\centering
\begin{tabular}{@{}lccc@{}}
\toprule
Quantity
& \texttt{hettich4}~\cite{hettich_1979}
& \texttt{hettich5}~\cite{hettich_1979}
& \texttt{leon10}~\cite{LEON200078} \\
\midrule
\(T\)
& \(300000\)
& \(10000\)
& \(10000\) 
\\
\(\bar\eta\)
& \(0.1\)
& \(0.1\)
& \(0.01\) \\
\(\bar\nu\)
& \(0.1\)
& \(0.1\)
& \(0.01\) \\
\bottomrule
\end{tabular}
\caption{Hyperparameters used for AGP on the benchmark problems in Section~\ref{ssec:BenchmarkSIPAMPL}.}
\label{tab:agp_gda_hyperparameters}
\end{table}

\subsection{Parameters of AGP and GDA for Security Strategy Computation in Section~\ref{ssec:MultiplayerGames}}
\label{subsec:security_strategy_parameters}

For the security strategy computation experiments in
Section~\ref{ssec:MultiplayerGames}, we use the same AGP parameters as those
reported in Section~\ref{ref:app_AGP_GDA_parameters_TE2}. For GDA, we use a
different set of schedules selected through a hyperparameter sweep:
\begin{align*}
\beta(t)
&=
\frac{1}{0.6}(t+1)^{0.98}, \\
\gamma(t)
&=
\frac{1}{0.9}(t+1)^{0.56}.
\end{align*}}

\bibliographystyle{siamplain}
\bibliography{references}

\end{document}